\documentclass[11pt]{amsart}
\usepackage{verbatim}
\usepackage{amsmath,amsfonts,amsthm,amsopn,cite,mathrsfs}
\usepackage{epsfig,verbatim}
\usepackage{subfigure}
\include{macros}

\usepackage{mathdots}

\usepackage[pdftex,
colorlinks, 
bookmarksnumbered, 
bookmarksopen, 
linkcolor=blue, 
citecolor=purple, 
urlcolor=orange, 
]{hyperref}

\usepackage[latin1]{inputenc}

\usepackage{latexsym}
\usepackage{mathabx}
\usepackage{calligra}

\usepackage{graphicx}
\usepackage{subfigure}
\usepackage[justification=centering]{caption}

\usepackage{xcolor}

\usepackage{doi}

\usepackage{url}

\newcommand\bcdot{\ensuremath{%
  \mathchoice%
   {\mskip\thinmuskip\lower0.2ex\hbox{\scalebox{1.5}{$\cdot$}}\mskip\thinmuskip}}%
   {\mskip\thinmuskip\lower0.2ex\hbox{\scalebox{1.5}{$\cdot$}}\mskip\thinmuskip}%
   {\lower0.3ex\hbox{\scalebox{1.2}{$\cdot$}}}%
   {\lower0.3ex\hbox{\scalebox{1.2}{$\cdot$}}}%
}

\newtheorem{thm}{Theorem}\numberwithin{thm}{section}

\newtheorem{prop}[thm]{Proposition}
\newtheorem{lem}[thm]{Lemma}
\newtheorem{cor}[thm]{Corollary}

\theoremstyle{definition}
\newtheorem{ex}[thm]{Example}
\newtheorem{exs}[thm]{Examples}
\newtheorem{rem}[thm]{Remark}

\newtheorem{dfn}[thm]{Definition}

\newtheorem{ass}[thm]{Assumptions}
\newtheorem{obs}[thm]{Observation}

\newcommand{\e}{\varepsilon}						
\renewcommand{\l}{\lambda}							
\newcommand{\nn}{\nonumber}								
\renewcommand{\asymp}{\sim}

\newcommand{\N}{\mathbb{N}}					
\newcommand{\Z}{\mathbb{Z}}						
\newcommand{\R}{\mathbb{R}}								
\newcommand{\CF}{\mathbb{C}}						

\newcommand{\BL}{\mathscr{L}}								
\newcommand{\MOP}[1]{\mathop{\mathrm{#1}}}						
\newcommand{\Op}[2]{\MOP{Op^{#1}}\left( #2 \right)}						
\newcommand{\Tr}{\mathrm{Tr}}
\newcommand{\vN}{\mathop{\mathrm{VN}}}							
\newcommand{\Sp}{\mathop{\mathrm{Sp}}}						
\newcommand{\tr}{\mathrm{\tau}}

\newcommand{\Abs}[1]{\left| #1\right|}						
\newcommand{\bracket}[2]{\left\langle #1 , #2\right\rangle}						
\newcommand{\Norm}[2]{\left\| #2\right\|_{#1}}								
\newcommand{\qn}[2]{\left | #2 \right |_{#1}}

\newcommand{\Lie}[1]{\frak{#1}} 		
\newcommand{\Rspan}[1]{\mathbb{R}\text{-}\mathrm{span}\{ #1\}} 			
\newcommand{\Liez}[1]{\mathfrak{z}(\mathfrak{#1})}						
\newcommand{\subgr}{\leq}									
\newcommand{\nsubgr}{\vartriangleleft}						
\newcommand{\dimG}{d}								
\newcommand{\pid}{\mathfrak{m}}					
\newcommand{\qa}{\mathfrak{h}}						
\newcommand{\PID}{M}									
\newcommand{\QA}{H}										
\newcommand{\rquo}{\setminus}									
\newcommand{\Orbit}{\mathcal{O}}								
\newcommand{\UEA}[1]{\mathfrak{u}\left( \Lie{#1} \right)}			
\newcommand{\coAd}{\mathop{\mathrm{Ad}^*}}				
\newcommand{\coad}{\mathop{\mathrm{ad}^*}}							
\newcommand{\hdim}{Q}							
\newcommand{\RO}{\mathcal{R}}							
\newcommand{\RF}{P}										
\newcommand{\hdeg}{\nu}											
\newcommand{\Leb}{m}										

\newcommand{\HS}{\mathcal{H}}			
\newcommand{\RS}{\mathcal{H}_\pi}				
\newcommand{\fd}{d_\pi}							
\newcommand{\Pf}{\mathrm{Pf}}						
\newcommand{\indR}[3]{\mathop{\mathrm{ind}^{#3}_{#2}(#1)}} 			
\newcommand{\rp}{\kappa}										
\newcommand{\om}{\mu}								
\newcommand{\Pla}{\mu}							

\renewcommand{\H}{\mathbf{H}_n}					
\newcommand{\h}{\mathfrak{h}_n}						
\newcommand{\HG}[2]{\mathbf{H}_{#1, #2}}						
\newcommand{\HA}[2]{\mathfrak{h}_{#1, #2}}							
\newcommand{\gv}[1]{\tilde{#1}}						

\renewcommand{\L}[2]{L^{#1}(#2)}					
\newcommand{\LW}[3]{L^{#1}_{#3}(#2)}					
\newcommand{\SF}[1]{\mathscr{S}(\mathbb{R}^{#1})}			
\newcommand{\SFG}[1]{\mathscr{S}(#1)}							
\newcommand{\SDG}[1]{\mathscr{S}'(#1)}			

\newcommand{\vt}{\mathbf{v}} 			
\newcommand{\pt}{\mathbf{p}} 					
\newcommand{\qt}{\mathbf{q}} 				
\newcommand{\st}{\mathbf{s}} 			

\newcommand{\SL}{\mathcal{L}}	
\newcommand{\HO}{\mathcal{Q}_{\R^n}}		
\newcommand{\HHO}{\mathcal{Q}_{\H}}			
\newcommand{\HAO}{\mathcal{A}_{\H}}					

\begin{document}
\title[Harmonic and Anharmonic Oscillators on $\H$]{Harmonic and Anharmonic Oscillators \\ on the Heisenberg Group}

\author[D. Rottensteiner]{David Rottensteiner}
\address{
David Rottensteiner:
\endgraf Faculty of Mathematics
\endgraf University of Vienna
\endgraf Oskar-Morgenstern-Platz 1, 1090 Vienna 
\endgraf
Austria}
\email{david.rottensteiner@univie.ac.at}

\author[M. Ruzhansky]{Michael Ruzhansky}
\address{
  Michael Ruzhansky:
  \endgraf
  Department of Mathematics: Analysis,
Logic and Discrete Mathematics
  \endgraf
  Ghent University, Belgium
  \endgraf
  and
  \endgraf
  School of Mathematical Sciences
    \endgraf
    Queen Mary University of London
  \endgraf
  United Kingdom
  \endgraf
  }
\email{ruzhansky@gmail.com}


\begin{abstract}
Although there is no canonical version of the harmonic oscillator on the Heisenberg group $\mathbf{H}_n$ so far, we make a strong case for a particular choice of operator by using the representation theory of the Dynin-Folland group $\mathbf{H}_{n, 2}$, a $3$-step stratified Lie group, whose generic representations act on $L^2(\mathbf{H}_n)$. Our approach is inspired by the connection between the harmonic oscillator on $\mathbb{R}^n$ and the sum of squares in the first stratum of $\mathbf{H}_n$ in the sense that we define the harmonic oscillator on $\mathbf{H}_n$ as the image of the sub-Laplacian $\mathcal{L}_{\mathbf{H}_{n, 2}}$ under the generic unitary irreducible representation $\pi$ of the Dynin-Folland group which has formal dimension $d_\pi = 1$. This approach, more generally, permits us to define a large class of so-called anharmonic oscillators by employing positive Rockland operators on $\mathbf{H}_{n, 2}$.
By using the methods developed in ter Elst and Robinson~\cite{tERo}, we obtain spectral estimates for the harmonic and anharmonic oscillators on $\mathbf{H}_n$. 

Moreover, we show that our approach extends to graded $SI/Z$-groups of central dimension $1$, i.e., graded groups which possess unitary irreducible representations which are square-integrable modulo the $1$-dimensional center $Z(G)$.

The latter part of the article is concerned with spectral multipliers. By combining ter Elst and Robinson's techniques with recent results in \cite{AkRu18}, we obtain useful $L^\mathbf{p}$-$L^\mathbf{q}$-estimates for spectral multipliers of the sub-Laplacian $\mathcal{L}_{\mathbf{H}_{n, 2}}$ and, in fact more generally, of general Rockland operators on general graded groups. As a by-product, we recover the Sobolev embeddings on graded groups established in \cite{FiRu17}, and obtain explicit hypoelliptic heat semigroup estimates.
\end{abstract}

\maketitle

\tableofcontents

\section{Introduction}

In this article we strongly propose a natural candidate for a canonical choice of the harmonic oscillator on the Heisenberg group $\H$. Our approach is motivated by the fact that the harmonic oscillator\footnote{The factor $4 \pi^2$ in the harmonic oscillator arises from our choice of realizing the Schr\"{o}dinger representation, but creates no essential deviation from the versions in Folland~\cite{FollPhSp} or Stein~\cite{Stein}.} $\HO := - \Delta + 4 \pi^2 \Abs{t}^2$ on $\R^n$, on the one hand, can be written as the image $-d\rho(\SL_{\H})$ of the negative sub-Laplacian on $\H$ under the infinitesimal Schr\"{o}dinger representation $d\rho$ of the Heisenberg Lie algebra $\h$, and, on the other hand, as the Weyl quantization (again related to $\rho$) of the symbol $\sigma(t, \xi) := 4 \pi^2 (\Abs{\xi}^2 + \Abs{t}^2)$, that is,
	\begin{align}
		- \Delta + 4 \pi^2 \Abs{t}^2 = -d\rho(\SL_{\H}) = \Op{w}{\sigma} = \iint_{\H/Z(\H)} \widehat{\sigma}(x,y) \rho(x,y) \, dx \, dy. \label{CriticalRelHO}
	\end{align}
Of all Lie groups whose unitary irreducible representations act on $\L{2}{\H}$ we therefore pick the $3$-step nilpotent Dynin-Folland group $\HG{n}{2}$, as it was conceived in Dynin's foundational paper~\cite{Dyn1} precisely to solve the problem of the Weyl quantization on $\H$, i.e., the momentum variables in $T^* \H \cong \HG{n}{2}/Z(\HG{n}{2})$ are mapped correctly onto the left-invariant vector fields on $\H$ while the coordinate functions are mapped correctly onto the coordinate multiplication operators. In particular this solves the analogue of \eqref{CriticalRelHO}.
The crucial ingredient for the Weyl quantization are the group's generic unitary irreducible representations.

Many years after Dynin's brief introduction of the group (without giving it a name), Folland~\cite{FollMeta} studied the group in detail as a special case of what he called meta-Heisenberg groups: the meta-Heisenberg group of the Heisenberg group. Recently, Fischer, Rottensteiner and Ruzhansky~\cite{FiRoRu} characterized the group's unitary irreducible representations in order to study the related coorbit spaces in the sense of Feichtinger and Gr\"{o}chenig's foundational paper~\cite{fegr89}. In recognition of Dynin's and Folland's works, the authors named the group the ``Dynin-Folland group'' and denoted it by $\HG{n}{2}$, in recognition of the meta-Heisenberg aspect.

In this article, we make use of the intimate connection between left-invariant\footnote{This is our choice. Equivalently, one could focus on right-invariant operators.} operators on $\H$ and the generic representations of its meta-Heisenberg group $\HG{n}{2}$ in order to define a canonical harmonic oscillator $\HHO$ on $\H$ as the image $-d\pi(\SL)$ of the negative sub-Laplacian $\SL_{\HG{n}{2}}$ under the generic representation $\pi \in \widehat{\mathbf{H}}_{n, 2}$ of formal dimension $\fd = 1$. Employing the machinery developed in ter Elst and Robinson~\cite{tERo}, we provide concrete spectral estimates for this operator and those dilated versions of it which correspond to the images of $-\SL_{\HG{n}{2}}$ under all the other generic representations $\pi \in \widehat{\mathbf{H}}_{n, 2}$. The spectral estimates are initially given for $\HG{n}{2}$ equipped with the canonical homogeneous structure, of which $\H$ equipped with its canonical homogeneous structure forms a subgroup. Since the techniques in \cite{tERo} were developed for positive Rockland operators on general graded groups, we extend our approach to study a much wider class of so-called anharmonic oscillators $\HAO$ on $\H$, which are defined as the images of positive Rockland operators $\RO$ under the generic representations $\pi \in \widehat{\mathbf{H}}_{n, 2}$. The spectral estimates we provide for these operators are given with respect to a large class of homogeneous structures on $\HG{n}{2}$. As a special case this includes the family of the harmonic oscillator and its natural dilates. Moreover, we show that the concrete spectral estimates extend from $\HG{n}{2}$ to a large class of graded groups which comprises $\H$ and $\HG{n}{2}$.
The last section of the article pays tribute to the prominent role the sub-Laplacians and positive homogeneous Rockland operators on $\HG{n}{2}$ play in our approach to harmonic and anharmonic oscillators on $\H$. By another application of ter Elst and Robinson's machinery combined with recent results established in Akylzhanov and Ruzhansky~\cite{AkRu18}, we prove asymptotic bounds for spectral multipliers of the sub-Laplacian $\SL_{\HG{n}{2}}$ and, more generally, for spectral multipliers of Rockland operators on general graded groups. Notably, we recover Sobolev embeddings, which were first proved in Folland~\cite{Fo75} for the sub-Laplacian on a generic stratified group and, more generally in Fischer and Ruzhansky~\cite{FiRu17}, for all positive Rockland operators on a general graded group.

Let us mention that we are not the first to propose a definition of harmonic oscillator on $\H$. Fischer~\cite{Fi11}, for example, makes a case for $\mathcal{Q}_{\mathcal{H}_1} := -d\pi(\SL_G)$ for the sub-Laplacian on a $6$-dimensional $2$-step nilpotent Lie group $G$. The main difference between this operator and the harmonic oscillator that we propose here lies in the quadratic potential added to the sub-Laplacian $\SL_{\mathcal{H}_1}$. However, the spectrum of her operator $\mathcal{Q}_{\mathcal{H}_1}$ in \cite{Fi11} is continuous whereas the spectrum of our version of $\mathcal{Q}_{\mathcal{H}_1}$ is discrete. Since the spectrum of the harmonic oscillator on $\R^1$, and more generally on $\R^n$, is discrete (in practice and also due to physical considerations), it is natural to expect that any reasonable candidate for $\mathcal{Q}_{\mathcal{H}_1}$ would have the same features, which is indeed the case with the operator considered in this paper.

As an example, with a general derivation in Section \ref{SectionHHO} and, more specifically, taking $\rp=1$ in Example \ref{EX:haH1}, our proposed harmonic oscillator $\mathcal{Q}_{\mathbf{H}_1}$ on the first Heisenberg group ${\mathbf{H}_1}$ with the canonical coordinates $(t_1,t_2,t_3)$ has the explicit form
	\begin{align*}
			\mathcal{Q}_{\mathbf{H}_1}= - \bigl( \partial_{t_1}^2 + \partial_{t_2}^2 \bigr) - \frac{1}{4} \bigl ({t_1}^2 + {t_2}^2 \bigr ) \hspace{1pt} \partial_{t_3}^2 + \bigl ( t_1 \hspace{1pt} \partial_{t_2} - t_2 \hspace{1pt} \partial_{t_1} \bigr ) \hspace{1pt} \partial_{t_3} + 4 \pi^2  \hspace{1pt} t_3^2.
\end{align*}
We refer to Subsection \ref{Subsection_HAO} for the details of many examples of 
different families of anharmonic oscillators on ${\mathbf{H}_n}$.

We note that in the case of $\R^n$ the general families of anharmonic oscillators and their spectral properties have been recently analyzed in \cite{ChDeRu18} from the point of view of the Weyl-H\"ormander theory. We refer to \cite{ChDeRu18} also for a more extensive discussion of the history of the results on the Euclidean anharmonic oscillators.

\smallskip
The article is organized as follows. In Section~\ref{ClassicalTheory} we recall in more detail the role of the Heisenberg group in \eqref{CriticalRelHO}. We provide the most crucial details about $\H$ and its generic representations $\rho \in \widehat{\mathbf{H}}_n$, realized as the Schr\"{o}dinger representations on $\L{2}{\R^n}$.

In Section~\ref{MachinerytERob} we briefly recall the machinery developed in ter Elst and Robinson~\cite{tERo} for positive Rockland operators on general graded groups. To this end, we recall the definitions of homogeneous structures and Rockland operators. The section also provides two examples for the usefulness of ter Elst and Robinson's results: spectral estimates for the harmonic oscillator $\HO$ and for a family of anharmonic oscillators on $\R$.

In Section~\ref{IntroDF} we give a quite brief but complete introduction to the Dynin-Folland $\HG{n}{2}$, its natural stratification, and its generic representations.

In Section~\ref{SectionHHO} we define the canonical harmonic oscillator $\HHO$ and provide a concrete formula in terms of the exponential coordinates on $\H$.

The main results of this paper are presented in Sections~\ref{SpecEst} to \ref{LpLq}. In Section~\ref{SpecEst} we explain why it is reasonable to extend the definition of $\HHO$ somewhat by employing all generic representations $\pi \in \widehat{\mathbf{H}}_{n, 2}$ and provide spectral estimates for this ``family of canonical harmonic oscillators'' in terms of the canonical homogeneous structure of $\HG{n}{2}$ related to its natural stratification. We then extend the approach to define a large class of anharmonic oscillators on $\H$ in terms of general positive Rockland operators on $\HG{n}{2}$. The subsequent spectral estimates for these operators are not restricted to the canonical homogeneous structure on $\HG{n}{2}$ anymore but provided for a large family of homogeneous structures, which allow the most precise estimates possible by ter Elst and Robinson's machinery. The harmonic oscillator $\HHO$ is included as a special case.

In Section~\ref{SecSpecEstSIZ} we prove that our approach via the Dynin-Folland group $\HG{n}{2}$ extends to all graded groups with $1$-dimensional center and unitary irreducible representations which are square-integrable modulo the center. The corresponding spectral estimates, however, are only given with respect to the canonical homogeneous structure for a given gradation since precise estimates require some a priori knowledge (about the dilations we use) which for an arbitrary homogeneous structure is not available.

Finally, in Section~\ref{LpLq} we show how the techniques used so far give rise to estimates on $L^\pt$-$L^\qt$-multipliers on the Dynin-Folland group and, somewhat surprisingly, on all graded groups. In particular, this holds for every admissible homogeneous structure on a given gradable group. The $L^\pt$-$L^\qt$-estimates are built on an explicit asymptotic growth-bound, which we can provide for every positive Rockland operator $\RO$ on a graded group equipped with an arbitrary, but fixed homogeneous structure. The bound itself is given in terms of the homogeneous degree of $\RO$ and the homogeneous dimension $\hdim$ of the group $G$. The actual $L^\pt$-$L^\qt$-estimates follow from an application of recent results on spectral multipliers on locally compact groups established in Akylzhanov and Ruzhansky~\cite{AkRu18}. It is worth noting that we recover important Sobolev embeddings due to Folland~\cite{Fo75} and Fischer and Ruzhansky~\cite{FiRu17}.

\section{The Harmonic Oscillator on $\R^n$} \label{ClassicalTheory}


In this section we briefly recall the most important facts about the relationship between the Heisenberg group $\H$ equipped with its natural stratification and the corresponding sub-Laplacian $\SL_{\H}$, and the canonical harmonic oscillator $\HO$ on $\R^n$. Since this is very well known matter, we will mainly use this section as a vehicle to set our notation and highlight specific details which will play a role in later sections. We will therefore omit the proofs and refer to the monograph Fischer and Ruzhansky~\cite{FiRuMon} instead.

\subsection{The Heisenberg Group and its Natural Stratification}

Our notation here will be chosen to be most convenient for the presentation of the results of this paper.

	\begin{dfn} \label{DefH}
Let $n \in \N$. We define the Heisenberg Lie algebra $\h$ to be $\R^{2n+1}$ equipped with the Lie bracket which for the standard basis $\{ X_{2n+1}, \ldots, X_1 \}$ is defined by the non-vanishing commutators 
$[X_j, X_{j+n}] = X_{2n+1}$ for all $j = 1, \ldots, n$.

We define the Heisenberg group $\H$ to be the connected, simply connected nilpotent Lie group obtained by exponentiating $\h$.
	\end{dfn}

By abbreviating elements $g = \exp_{\H} \bigl( x_{2n+1} X_{2n+1} + \ldots + x_1 X_1 \bigr) \in \H$ by $(x_{2n+1}, x_{2n}, \ldots, x_1)$ the $\H$-group law in exponential coordinates is given by
	\begin{align}
		 (x_{2n+1}, x_{2n}, \ldots, x_1) (x'_{2n+1}, &x'_{2n}, \ldots, x'_1) \nn \\ &= \Bigl( x_{2n+1} + x'_{2n+1} + \frac{1}{2} \sum_{j=1}^n (x_j x'_{n+j} - x'_j x_{n+j}), x_{2n} + x'_{2n}, \ldots, x_1 + x'_1 \Bigr).  \label{GrLawH}
	\end{align}
Often it is convenient to group the variables as $\gv{x}_3 := x_{2n+1}, \gv{x}_2 := (x_{2n}, \ldots, x_{n+1}), \gv{x}_1 := (x_n, \ldots, x_1)$ and rewrite the group law as
	\begin{align}
		(\gv{x}_3, \gv{x}_2, \gv{x}_1) (\gv{x}'_3, \gv{x}'_2, \gv{x}'_1) = \Bigl( \gv{x}_3 + \gv{x}'_3 + \frac{1}{2} \bigl( \bracket{\gv{x}_1}{\gv{x}'_2} - \bracket{\gv{x}_2}{\gv{x}'_1} \bigr), \gv{x}_2 + \gv{x}'_2, \gv{x}_1 + \gv{x}'_1 \Bigr). \label{GrLawHgv}
	\end{align}

The Heisenberg Lie algebra $\h$ possesses a natural stratification together with a family of dilations. 

	\begin{lem}
The Heisenberg Lie algebra $\h$ admits a stratification $\h = \Lie{g}_3 \oplus \Lie{g}_2 \oplus \Lie{g}_1$ with
        \begin{align*}
        	\Lie{g}_3 := \R X_{2n+1}, \hspace{10pt} \Lie{g}_2 := \Rspan{X_{2n}, \ldots, X_{n+1}}, \hspace{10pt} \Lie{g}_1 := \Rspan{X_n, \ldots, X_1},
        \end{align*}
for which
	\begin{align*}
		D_r(X_{2n+1}) = r^2 X_{2n+1}, \hspace{10pt} D_r(X_{j+n}) = r X_{j+n}, \hspace{10pt} D_r(X_j) = r X_j \hspace{10pt} j= 1, \ldots, n,
	\end{align*}
defines a family of dilations on $\h$.
	\end{lem}

Due to the stratification, the negative sum of squares of squares $\bigl( X_1^2 + X_2^2 \bigr) \in \UEA{\h}$ gives rise to a hypoelliptic sub-Laplacian $\SL_{\H} := dR \bigl( X_1^2 + X_2^2 \bigr)$, which is densely defined on $\L{2}{\H}$; here $R$ denotes the right regular representation $R$ of $\H$ and $dR$ the extension to the universal enveloping algebra $\UEA{\h}$ of its derivative. The sub-Laplacian is negative definite and possesses a self-adjoint extension to $\L{2}{\H}$. In the above exponential coordinates it is explicitly given by
	\begin{align*}
		\SL_{\H} &= \bigl( \partial_{t_1} - \frac{1}{2} t_{n+1} \partial_{t_{2n+1}} \bigr)^2 + \ldots + \bigl( \partial_{t_n} - \frac{1}{2} t_{2n} \partial_{t_{2n+1}} \bigr)^2 \\
		& \hspace{10pt} + \bigl( \partial_{t_{n+1}} + \frac{1}{2} t_1 \partial_{t_{2n+1}}\bigr)^2 + \ldots + \bigl( \partial_{t_{2n}} + \frac{1}{2} t_n \partial_{t_{2n+1}}\bigr)^2.
	\end{align*}

\subsection{The Harmonic Oscillator Via The Schr\"{o}dinger Representation}

In this subsection we recall the connections between the sub-Laplacian of Heisenberg group $\SL_{\H}$ and the harmonic oscillator on $\R^n$, also known as the Hermite operator. To this end, we recall that the Heisenberg group $\H$ possesses a family of unitary irreducible representations which are square-integrable modulo the center $Z(\H) = \exp_{\H} \bigl( \R X_{2n+1} \bigr)$. Each of these representations corresponds to a uniquely determined flat coadjoint orbit
	\begin{align*}
		\Orbit_{\rp X^*_{2n+1}} = \rp X^*_{2n+1} + \Rspan{X^*_{2n}, \ldots, X^*_1}
	\end{align*}
with $\rp \in \R \setminus \{ 0 \}$. Accordingly, we denote the representations by $\rho_{\rp}$ and call them the Schr\"{o}dinger representations of $\H$ of parameter $\rp \in \R \setminus \{ 0 \}$. Since the subalgebra $$\pid := \Rspan{X_{2n+1}, X_{2n}, \ldots, X_{n+1}} \nsubgr \h$$ is polarizing for all representatives $\rp X^*_{2n+1}$, $\rp \in \R \setminus \{ 0 \}$, and $\R^n \cong \PID \rquo \H$ for $\PID := \exp_{\H} (\pid)$, each $\rho_{\rp}$ can be realized in $\HS_{\rho_{\rp}} = \L{2}{\R^n}$ as the representation
$\rho_{\rp} := \indR{\chi_{\rp X^*_{2n+1}}}{\PID}{\HG{n}{2}}$ induced by the character
	\begin{align*}
		\chi_{\rp X^*_{2n+1}}: \PID \to \CF: m \mapsto e^{2 \pi i \bracket{\rp X^*_{2n+1}}{\log(m)}}.
	\end{align*}
Expressed in the exponential coordinates of \eqref{GrLawH}, the representation $\rho_{\rp}$ acts on $f \in \L{2}{\R^n}$ by
	\begin{align}
		\bigl( \rho_{\rp}(x_{2n+1}, \ldots, x_1)f \bigr)(t_n, \ldots, t_1) = e^{2 \pi i \rp \bigl( x_{2n+1} + \frac{1}{2} \sum_{j 1}^n x_j (x_{j+n} + 2 t_j) \bigr)} \hspace{1pt} f(t_n + x_n, \ldots, t_1 + x_1). \label{SchrRep}
	\end{align}
If we group the variables as in \eqref{GrLawHgv} and, accordingly, set $t := (t_n, \ldots, t_1)$, this identity shortens to
	\begin{align}
		\bigl( \rho_{\rp}(\gv{x}_3, \gv{x}_2, \gv{x}_1)f \bigr)(t) = e^{2 \pi i \rp \bigl( \gv{x}_3 + \frac{1}{2} \bracket{\gv{x}_1}{\gv{x}_2} + \bracket{\gv{x}_2}{t} \bigr)} \hspace{1pt} f(t + \gv{x}_1). \nn 
	\end{align}
The action of the infinitesimal representation $d\rho_{\rp}$ of $\h$ can be written out for the smooth vectors $\HS^\infty_{\rho_{\rp}} \cong \SF{n}$ of $\rho_{\rp}$:
	\begin{equation} \label{InfRepH}
	\left\{ \begin{array}{rcl}
		d\pi_{\rp}(X_{2n+1}) f  &=& 2 \pi i \rp \hspace{2pt} f, \\
		d\pi_{\rp}(X_{j+n}) f &=& 2 \pi i  \rp \hspace{2pt}  t_j \hspace{1pt} f, \hspace{80pt} j = 1, \ldots, n. \\
		d\pi_{\rp}(X_j) f &=& \partial_{t_j} f,
	\end{array}\right. \nn
	\end{equation}
The latter two identities indicate the strong connection with the harmonic oscillator, which
following the conventions of Stein's monograph~\cite{Stein}, is defined as the positive definite essentially self-adjoint operator on $\L{2}{\R^n}$ given by
	\begin{align*}
		- \Delta + \Abs{t}^2 = \sum_{j = 1}^n (- \partial_{t_j}^2 + t_j^2 ).
	\end{align*}
We recall that this operator has discrete spectrum of eigenvalues $\l_s = (2 \Abs{s} + n)$ for $s \in \Z^n$, and the Hermite functions form an eigenbasis of $\L{2}{\R^n}$; for a proof we refer to Stein~\cite[Ch.~XII~\S~6]{Stein} or Folland~\cite[Ch.~1~\S~7]{FollPhSp}. However, the factors $(2 \pi i)$ in \eqref{InfRepH} disturb the connection between the sum of squares $\SL_{\H} \in \UEA{\h}$ and the harmonic oscillator. At the modest price of a slight rescaling, yet retaining all the interesting properties, we modify the definition to fit within our framework.

	\begin{dfn} \label{DefHO}
The harmonic oscillator on $\R^n$ is defined to be the positive definite essentially self-adjoint operator $\HO$ on $\L{2}{\R^n}$ given by
	\begin{align*}
		\HO := - \Delta + 4 \pi^2 \hspace{1pt} \Abs{t}^2 = \sum_{j = 1}^n (- \partial_{t_j}^2 + 4 \pi^2 t_j^2 ).
	\end{align*}
	\end{dfn}

With this definition at hand, we can formulate the connection between $\SL_{\H}$ and $\HO$ and summarize some important properties arising from it. For a proof we again refer to \cite{FollPhSp} or \cite{Stein}.

	\begin{prop}
Let $\bigl\{ \rho_{\rp} \bigr \}_{\rp \in \R \setminus \{ 0 \}}$ be the Schr\"{o}dinger representations of the Heisenberg group $\H$, realized as in \eqref{SchrRep}, and let $\rho = \rho_1$. Let $\HO$ be the harmonic oscillator on $\R^n$ fixed by Definition~\ref{DefHO}. Then for $\RF := - \sum_{j = 1}^{2n} X^2_j$ we have 
	\begin{align*}
		- \SL_{\H} &= dR(\RF), \\
		\HO &= d\rho(\RF).
	\end{align*}
For $\rp \in \R \setminus \{ 0 \}$ the operators
	\begin{align*}
		\HO^\rp := d\rho_{\rp}(\RF) = - \Delta + 4 \pi^2 \rp^2 \hspace{1pt} \Abs{t}^2
	\end{align*}
are positive definite and essentially self-adjoint on $\L{2}{\R^n}$ with discrete spectrum. Their eigenvalues are given by 
	\begin{align}
		\l_{s, \rp} = 2 \pi \Abs{\rp} (2 \Abs{s} + n), \hspace{20pt} s \in \Z^n. \label{EigValHO}
	\end{align}
	\end{prop}

\section{On Positive Rockand Operators and Their Spectra} \label{MachinerytERob}

In this section we recall one of the main results in ter Elst and Robinson~\cite{tERo} on spectral estimates for operators whose principal part is a positive Rockland form. To this end, we recall the basic definitions in the context of homogeneous Lie groups and Rockland operators. For a detailed exposition of homogeneous, in particular graded, groups and Rockland operators, we refer to our main reference \cite{FiRuMon}, in particular to the Chapters~3 and 4.

	\begin{dfn} \label{DefHomGr}
(i) A family of dilations $\{ D_r \}_{r > 0}$ of a Lie algebra $\Lie{g}$ is a family of Lie algebra automorphisms on $\Lie{g}$ such that
	\begin{align}
		D_r := \exp_{GL(n, \R)} \bigl( A \log(r) \bigr) \label{Dilations}
	\end{align}
for a diagonalizable linear map $A$ on $\Lie{g}$ with positive eigenvalues $\vt_1, \ldots, \vt_{\dim(\Lie{g})}$.

(ii) A homogeneous group is a connected, simply connected Lie group whose Lie algebra is equipped with dilations.

(iii) The eigenvalues of $A$ will be referred to as the weights of $\{ D_r \}_{r > 0}$. The set of weights is denoted by $\mathcal{W}_A$ and the eigenspace of a weight $\vt \in \mathcal{W}_A$ is denoted by $W_\vt$.

(vi) Given a Lie algebra $\Lie{g}$ equipped with dilations $\{ D_r \}_{r > 0}$, we will denote by $\{ D^*_r \}_{r > 0}$ the dilations on the dual Lie algebra $\Lie{g}^*$ obtained by duality.
	\end{dfn}

Let us recall some well-known facts: Every homogeneous group is nilpotent and the dilations $\{ D_r \}_{r > 0}$ lift to the Lie group $G$ via exponential or Malcev coordinates. Every graded group is homogeneous and in this case one can always rescale the dilations' weights so that they are positive integers with one as their greatest common divisor. One then defines the homogeneous dimension of $G$ to be $\hdim := \vt_1+ \ldots + \vt_{\dim(\Lie{g})} = \Tr(A)$. Moreover, every graded group permits a canonical homogeneous structure defined as follows.

	\begin{dfn} \label{CanDilGrGr}
Let $\Lie{g}$ be a graded Lie algebra with gradation $\Lie{g} = \oplus_{k = 1}^N \Lie{g}_k$. Let $r > 0$ and let $A$ be the matrix for which $A X = k X$ for all $X \in \Lie{g}_k$ and $k =1, \ldots, N$. Then we define the canonical dilations associated with the gradation of $\Lie{g}$ by \eqref{Dilations}.
	\end{dfn}

This definition implies that every basis given as the union of bases $\{ X_{k_1}, \ldots, X_{k_{\dimG_k}} \}$ of the direct summands $\Lie{g}_k$ is an eigenbasis of $A$ and $\mathcal{W}_A = \{ 1, \ldots, k \}$. Such a basis is, in particular, a strong Malcev basis for $\Lie{g}$.  A generic family of dilations, however, may have an eigenbasis which is not of this type; eigenvectors for a given weight may, for example, not belong to any direct summand.

Another crucial notion are the so-called homogeneous quasi-norms, which serve as the natural substitute for norms in the sense that they respect the homogeneous group structure.

	\begin{dfn} \label{DefHQN}
A homogeneous quasi-norm on a homogeneous Lie group G is a continuous non-negative function $G \ni g \mapsto \qn{G}{g} \in [0, \infty)$ such that for all $g \in G, r > 0$, we have
	\begin{itemize}
		\item[(i)] $\qn{G}{g^{-1}} = \qn{G}{g}$,
		\item[(ii)] $\qn{G}{D_r(g)} = r \qn{G}{g}$,
		\item[(iii)] $\qn{G}{g} = 0$ if and only if $g$ equals the unit element of $G$.
	\end{itemize}
	\end{dfn}

Ter Elst and Robinson's analysis relies on the use of the homogeneous norm on the dual Lie algebra $\Lie{g}^*$ of a given homogeneous group $G$ which is defined by
	\begin{align}
		\qn{\Lie{g}^*}{l} := \inf \Bigl \{ r > 0 : \bigl \| D^*_{1/r}(l) \bigr \|_{\Lie{g}^*} \leq 1 \Bigr \}, \label{defaultQN}
	\end{align}
where $\Norm{\Lie{g}^*}{\, . \,}$ is the dual norm on $\Lie{g}^*$ of a given norm $\Norm{\Lie{g}}{\, . \,}$ on $\Lie{g}$. However, it may occasionally be more convenient to work with homogeneous quasi-norms on $\Lie{g}^*$ which are equivalent but different from $\qn{\Lie{g}^*}{\, . \,}$. To this end, we give the following definition.

	\begin{dfn} \label{DefHQNdual}
A homogeneous quasi-norm on the dual Lie algebra $\Lie{g}^*$ of a homogeneous Lie group G is a continuous non-negative function $\Lie{g}^* \ni l \mapsto \qn{\Lie{g}^*}{l} \in [0, \infty)$ such that for all $l \in \Lie{g}^*, r > 0$, we have
	\begin{itemize}
		\item[(i)] $\qn{\Lie{g}^*}{l + l'} \leq C ( \qn{\Lie{g}^*}{l} + \qn{\Lie{g}^*}{l'} )$ for some $C \geq 1$,
		\item[(ii)] $\qn{\Lie{g}^*}{D^*_r(l)} = r \qn{\Lie{g}^*}{l}$,
		\item[(iii)] $\qn{\Lie{g}^*}{l} = 0$ if and only if $l = 0$.
	\end{itemize}
	\end{dfn}

It is well known that all homogeneous quasi-norms on $G$ are equivalent, and it is easy to see that all homogeneous quasi-norms on $\Lie{g}^*$ for a given family of dilations are equivalent, in particular they are equivalent to the quasi-norm defined by \eqref{defaultQN}; therefore all equivalent quasi-norms on $\Lie{g}^*$ induce the Euclidean topology.
Throughout this article we will use a very convenient type of homogeneous quasi-norm on $\Lie{g}^*$, which facilitates our computations enormously.

	\begin{prop} \label{QNGrGr}
Let $G$ be a graded group of topological dimension $\dimG$ equipped with a family of dilations $\{ D_r \}_{r > 0}$ defined by a weight matrix $A$. Let $\hdim$ be the associated homogeneous dimension of $G$ and let $\{ D^*_r \}_{r > 0}$ be the dilations on $\Lie{g}^*$ defined by duality. Let $\vt_1, \ldots, \vt_M$ be the distinct elements of the set of weights $\mathcal{W}_A$. Given an eigenbasis $\{ X_\dimG, \ldots, X_1 \} = \{ X_{M_{\dimG_M}}, \ldots, X_{M_1}, \ldots, X_{1_{\dimG_1}}, \ldots, X_{1_1} \}$, let us express all elements $l \in \Lie{g}^*$ in terms of the dual basis, i.e., $l = l_1 X^*_1 + \ldots + l_\dimG X^*_\dimG$. Then the map $\qn{\Lie{g}^*, A}{\, . \,}^\infty: \Lie{g}^* \to [0, \infty)$ defined by
	\begin{align*}
		\qn{\Lie{g}^*, A}{l}^\infty :=& \max \Bigl \{ \bigl | l_{k_{\dimG_k}} \bigr |^{1/\vt_k}, \ldots, \bigl | l_{k_1} \bigr |^{1/\vt_k} : k = 1, \ldots, M \Bigr \} \\
						&= \max \Bigl \{ \bigl | l_{M_{\dimG_M}} \bigr |^{1/\vt_M}, \ldots, \bigl | l_{M_1} \bigr |^{1/\vt_M}, \ldots, \bigl | l_{1_{\dimG_1}} \bigr |^{1/\vt_1}, \ldots, \bigl | l_{1_1} \bigr |^{1/\vt_1} \Bigr \}
	\end{align*}
is a homogeneous quasi-norm.
	\end{prop}

	\begin{proof}
The proof is quite straightforward. To show (i) of Definition~\ref{DefHQNdual}, we find that
	\begin{align*}
		\qn{\Lie{g}^*, A}{l + l'}^\infty &= \max \Bigl \{ \bigl | l_{k_{\dimG_k}} \bigr |^{1/{\vt_k}}, \ldots, \bigl | l_{k_1} \bigr |^{1/\vt_k} : k = 1, \ldots, M \Bigr \} \\
		& \leq \max \Bigl \{ C_{1/{\vt_k}} \bigl ( \bigl | l_{k_{\dimG_k}} \bigr |^{1/\vt_k} + \bigl | l'_{k_{\dimG_k}} \bigr |^{1/\vt_k} \bigr ), \ldots, C_{1/k} \bigl ( \bigl | l_{k_1} \bigr |^{1/\vt_k} + \bigl | l'_{k_1} \bigr |^{1/\vt_k} \bigr ) : k = 1, \ldots, M  \Bigr \}
	\end{align*}
for $M-1$ constants $C_{1/M}, \ldots C_{1/2} > 0$ and $C_1 := 1$. If we set $C := \max \{ C_{1/M}, \ldots, C_1 \}$, then
	\begin{align*}
		\qn{\Lie{g}^*, A}{l + l'}^\infty &\leq C \max \bigl \{ \bigl | l_{M_{\dimG_M}} \bigr |^{1/{\vt_M}} + \bigl | l'_{M_{\dimG_M}} \bigr |^{1/{\vt_M}}, \ldots, \bigl | l_{1_1} \bigr |^ {1/\vt_1}+ \bigl | l'_{1_1} \bigl |^{1/\vt_1} \bigr \} \\
		& \leq C \bigl( \qn{\Lie{g}^*, A}{l}^\infty + \qn{\Lie{g}^*, A}{l'}^\infty \bigr).
	\end{align*}
The properties (ii) and (iii) are obviously satisfied.
	\end{proof}

The canonical dilations associated to a given gradation of a gradable Lie algebra (cf.~Definition~\ref{CanDilGrGr}) give rise to a specific example of $\qn{\Lie{g}^*, A}{\, . \,}^\infty$.

	\begin{cor} \label{CQNGrGr}
Let $G$ be a graded group of topological dimension $\dimG$ with gradation $\Lie{g} = \oplus_{k = 1}^N \Lie{g}_k$ and set $\dimG_k :=\dim(\Lie{g}_k)$ for $j = 1, \ldots, N$. Let $G$ be equipped with the canonical homogeneous structure from Definition~\ref{CanDilGrGr} and let $\{ D^*_r \}_{r > 0}$ be the dilations on $\Lie{g}^*$ defined by duality. Given a strong Malcev basis $\{ X_\dimG, \ldots, X_1 \} = \{ X_{N_{\dimG_N}}, \ldots, X_{N_1}, \ldots, X_{1_{\dimG_1}}, \ldots, X_{1_1} \}$ passing through the gradation, let us express all elements $l \in \Lie{g}^*$ in terms of the dual basis, i.e., $l = l_1 X^*_1 + \ldots + l_\dimG X^*_\dimG$. Then the map $\qn{\Lie{g}^*}{\, . \,}^\infty: \Lie{g}^* \to [0, \infty)$ defined by
	\begin{align*}
		\qn{\Lie{g}^*}{l}^\infty :=& \max \Bigl \{ \bigl | l_{k_{\dimG_k}} \bigr |^{1/k}, \ldots, \bigl | l_{k_1} \bigr |^{1/k} : k = 1, \ldots, N \Bigr \} \\
						&= \max \Bigl \{ \bigl | l_{N_{\dimG_N}} \bigr |^{1/N}, \ldots, \bigl | l_{N_1} \bigr |^{1/N}, \ldots, \bigl | l_{1_{\dimG_1}}  \bigr |, \ldots, \bigl | l_{1_1} \bigr | \Bigr \}
	\end{align*}
is a homogeneous quasi-norm.
	\end{cor}


The rest of this section focuses on Rockland operators and the crucial results from ter Elst and Robinson~\cite{tERo} we will use throughout the article. Note that our definition follows \cite{tERo} and thus varies slightly from the equivalent Definition~4.1.2 in \cite{FiRuMon}.

	\begin{dfn} \label{DefRock}
Let $G$ be a homogeneous Lie group and let $\RF \in \UEA{\Lie{g}}$ be homogeneous of positive degree. We say that $\RF$ is a ``Rockland form'' if for every non-trivial unirrep $\pi \in \widehat{G}$ the operator $d\pi(\RF)$ is injective on the smooth vectors $\RS^\infty$, that is,
	\begin{align}
		\forall v \in \RS^\infty: \hspace{10pt} d\pi(\RF)v = 0 \Rightarrow v = 0. \label{RockPr}
	\end{align}
A ``Rockland operator'' on $G$ is a left-invariant differential operator $\RO$ with $\RO = dR(\RF)$ for some Rockland form $\RF \in \UEA{\Lie{g}}$, that is, an operator which is the image under the infinitesimal right regular representation of some Rockland form $\RF$.
	\end{dfn}

We recall the main tools for our analysis, \cite[Thm.~4.1]{tERo}:

	\begin{thm} \label{ThmtERo}
Let $G$ be a homogeneous Lie group and let $\Lie{g}^*$ be equipped with a homogeneous norm. Let $\pi \in \widehat{G}$ with coadjoint orbit $\Orbit_\pi$. Let $\om_\pi$ be the orbital measure on $\Orbit_\pi$ and for $\l > 0$ let
	\begin{align*}
		N_0(\l, \pi) := \om_\pi \bigl( \bigl \{ l \in \Orbit_\pi : \qn{\Lie{g}^*}{l} \leq \l \bigr \} \bigr).
	\end{align*}
Moreover, let $\RF$ be a positive Rockland form of homogeneous degree $\hdeg$ and let $N(\l, \pi, \RF)$ be the number of eigenvalues of $d\pi(\RF)$, counted with multiplicities, which are less or equal to $\l$. Then there exists a constant $c > 0$ such that
	\begin{align}
		c^{-1} \hspace{1pt} N_0(\l, \pi) \leq N(\l^\hdeg, \pi, \RF) \leq c \hspace{1pt} N_0(\l, \pi) \label{SpecAsymp}
	\end{align}
uniformly for all $\l > 0$ and all unirreps $\pi \in \widehat{G}$.
	\end{thm}


Since any two quasi-norms on $\Lie{g}^*$ equipped with a homogeneous structure are equivalent, the asymptotic distribution of eigenvalues $N(\l, \pi, \RF')$ can be computed by using any quasi-norm. In each concrete example of this article, we choose the most convenient one.

The following two examples show how efficient ter Elst and Robinson's techniques are. The first example discusses the harmonic oscillator on $\R^n$, the second one a family of anharmonic oscillators on $\R$, which was treated in \cite{tERo}. Although the family of anharmonic oscillators includes the harmonic oscillator, the argument for a specific anharmonic case is hardly more involved than the harmonic case. For this reason and the sake of clarity, we treat the first example quite explicitly in $n$ dimensions and only sketch the anharmonic oscillators in one dimension.

	\begin{ex} [The Harmonic Oscillator on $\R^n$] \label{ExHORn}
Although the eigenvalues \eqref{EigValHO} of the harmonic oscillator $\mathcal{Q}_{\R^n}$ are very well known, we could, in ignorance of their precise values, apply ter Elst and Robinson's method right away to estimate their asymptotic distribution. Since for the Schr\"{o}dinger representation $\rho = \rho_1$ on $\L{2}{\R^n}$ the positive Rockland form $\RF := - \sum_{j = 1}^{2n} X^2_j \in \UEA{\h}$ satisfies 
	\begin{align*}
		dR(\RF) &= - \SL_{\H}, \\
		d\rho(\RF) &= \mathcal{Q}_{\R^n},
	\end{align*}
we can apply Theorem~\ref{ThmtERo} and eventually obtain
	\begin{align}
		N(\l, \rho, \RF) \asymp \l^n \hspace{20pt} \mbox{ and } \hspace{20pt} \l_s \asymp s^{\frac{1}{n}}, \hspace{5pt} s = 1, 2, \ldots, \label{SpecEstHO}
	\end{align}
which is precisely what we asymptotically have for the eigenvalues $\l_s = 2 \Abs{s} + n$ of $\mathcal{Q}_{\R^n}$.

The explicit argument goes as follows. On each coadjoint orbit $\Orbit_{\rp X_{2n+1}^*}, \rp \in \R \setminus \{ 0 \},$ the orbital measure is given by $\om_{\rho_\rp} = \Abs{\rp}^{-n} \hspace{1pt} d\Leb_{\R^{2n}}$, that is, a weighted version of the $2n$-dimensional Lebesgue measure on the affine subspace $\Orbit_{\rp X_{2n+1}^*} \subseteq \h^*$. This follows directly from the direct integral representation $\Leb_{\h^*} = \int^\oplus_{\Liez{\h}^*} \om_{\rho_\rp} d\Pla(\rho_{\rp})$ of the Lebesgue measure on $\h^*$ by the Plancherel theorem (cf.~\cite[Thm.~4.3.10]{CoGr}) and the fact that the Plancherel measure on $\Liez{\h}^* \cong \h^*/\coAd(\H) \cong \widehat{\mathbf{H}}_n$ is given by $d\Pla(\rho_{\rp}) = \Abs{\rp}^n \hspace{1pt} d\rp$. So, for $\rho = \rho_1$ we have $\om_{\rho} = \Leb_{\R^{2n}}$. Then by an easy calculation or Proposition~\ref{CQNGrGr} one has that
$\qn{\h^*}{\, . \,}^\infty: \h^* \to [0, \infty)$ defined by
	\begin{align*}
		\qn{\h^*}{l}^\infty := \max \bigl \{ \Abs{l_{x_{2n+1}}}^{1/2}, \Abs{l_{x_{2n}}}, \ldots, \Abs{l_{x_1}} \bigr \}
	\end{align*}
is a homogeneous quasi-norm for the canonical homogeneous structure associated to the natural stratification of $\h$.
If the orbit $\Orbit_{\rp X_{2n+1}^*} = \rp X^*_{2n+1} + \Rspan{X^*_{2n}, \ldots, X^*_1}$ and the closed $\qn{\h^*}{\, . \,}^\infty$-neighborhood of radius $\l^{1/2}$ have non-void intersection, that is, if $\l^{1/2} \leq \rp$, then
	\begin{align*}
		\Leb_{\R^{2n}} \bigl( \bigl \{ l \in \Orbit_{\rp X_{2n+1}^*} : \qn{\Lie{h}_n^*}{l} \leq \l^{1/2} \bigr \} \bigr) &= \Leb_{\R^{2n}} \bigl( \bigl \{ l \in \Orbit_{\rp X_{2n+1}^*} : \Abs{l_{x_{2n+1}}} \leq \l, \Abs{l_{x_{2n}}} \leq \l^{1/2}, \ldots, \Abs{l_{x_1}} \leq \l^{1/2} \bigr \} \bigr ) \\
		&= 2^{2n} \l^{2n/2};
	\end{align*}
otherwise we have
	\begin{align*}
		\Leb_{\R^{2n}} \bigl( \bigl \{ l \in \Orbit_{\rp X_{2n+1}^*} : \qn{\Lie{h}_n^*}{l} \leq \l^{1/2} \bigr \} \bigr) = 0.
	\end{align*}
Thus, we have
	\begin{align*}
		N(\l, \rho_{\rp}, \RF) \asymp N_0(\l^{1/2}, \rho_{\rp}) 
		\asymp
		 \left\{ \begin{array}{lcr}
		  \l^n &\mbox{ if }& \l^{1/2} \leq \rp, \\
		  0 &\mbox{ otherwise.}&
		  \end{array} \right.
	\end{align*}
The second estimate in \eqref{SpecEstHO} follows from the observation that for small enough $\e > 0$ and large enough $s \in \N$ there exists a constant $c > 0$ such that for all $s \in \N$ we have
	\begin{align*}
		c^{-1} \hspace{1pt} (\l_s - \e)^n \leq N(\l_s - \e, \rho_{\rp}, \RF) \leq s \leq N(\l_s, \rho_{\rp}, \RF) \leq c \hspace{1pt} \l^n_s,
	\end{align*}
or equivalently, $\l_s \asymp s^{\frac{1}{n}}$. In particular, this holds for $\rp = 1$, which implies \eqref{SpecEstHO}.
	\end{ex}

	\begin{ex} [Anharmonic Oscillators on $\R$] \label{ExAO}
The example of the following family of anharmonic oscillators on $\R$ was treated in \cite{tERo}. For the homogeneous structure on $\Lie{h}_1$ defined by the modified dilations
	\begin{align*}
		D_r(X_3) = r^{\vt_1 + \vt_2} X_3, \hspace{10pt} D_r(X_2) = r^{\vt_2} X_2, \hspace{10pt} D_r(X_1) = r^{\vt_1} X_1 \hspace{10pt} \mbox{ for fixed } \vt_1, \vt_2 \in \N,
	\end{align*}
the form $\RF := (-1)^{\vt_2} \hspace{1pt} X_1^{2 \vt_2} + (-1)^{\vt_1} \hspace{1pt} X_2^{2 \vt_1}$ is positive and Rockland of homogeneous degree $\hdeg = 2 \vt_1 \vt_2$. By employing a convenient homogeneous quasi-norm on $\Lie{h}_1^*$, e.g., the one from Proposition~\ref{QNGrGr}, the same argument as above yields 
	\begin{align*}
		N(\l, \rho, \RF) \asymp \l^{(\vt_1 + \vt_2)/2 \vt_1 \vt_2} \hspace{20pt} \mbox{ and } \hspace{20pt} \l_s \asymp s^{2 \vt_1 \vt_2/(\vt_1 + \vt_2)}, \hspace{5pt} s = 1, 2, \ldots.
	\end{align*}
The spectral asymptotics of this type of anharmonic oscillator with additional polynomial potential were established in Helffer and Robert~\cite{HeRo82a, HeRo82b}. A recent article by Chatzakou, Delgado and Ruzhansky~\cite{ChDeRu18} investigates the spectral properties for negative powers of anharmonic oscillators in terms of Schatten-von Neumann classes using the Weyl-H\"{o}rmander theory.
	\end{ex}

\section{The Dynin-Folland Group $\HG{n}{2}$} \label{IntroDF}

The $3$-step connected, simply connected nilpotent Lie group we call the Dynin-Folland group $\HG{n}{2}$ was introduced in Dynin~\cite{Dyn1} in order to study a natural type of Weyl quantization of pseudo-differential operators on the Heisenberg 
group $\H$ and was studied in detail in Folland~\cite{FollMeta} as a specific example of the so-called
meta-Heisenberg groups. Recently, Fischer, Rottensteiner and Ruzhansky~\cite{FiRoRu} characterized the coorbit spaces related to $\HG{2}{n}$.

The group's generic representations and the homogeneous structure related to the natural stratification permit a natural notion of harmonic oscillator on $\H$ and give access to the asymptotic distribution of the oscillator's spectrum.

Note that our general notation was chosen in accordance with Corwin and Greenleaf's monograph~\cite{CoGr}. Our references for homogeneous, graded and stratified Lie groups are the monographs Folland and Stein~\cite{FoSt} and Fischer and Ruzhansky~\cite{FiRuMon}

\subsection{Dynin's Original Lie Algebra} \label{Dynin'sLieAlg}

In the case of the lowest dimension the Dynin-Folland group $\HG{1}{2}$ is a semi-direct
product $\HG{1}{2} = \R ^4 \rtimes \mathbf{H}_1$ of $\R^4$ and the $3$-dimensional Heisenberg group $\mathbf{H}_1$.
Its $3$-step nilpotent Lie algebra $\HA{1}{2}$ is defined by a strong Malcev basis $\{ Z, Y_1, Y_2, Y_3, X_3, X_2,
X_1 \}$ with Lie brackets
	\begin{align}
	\left[\begin{array}{c|ccc|ccc}
		[\,.\,, \,.\,]&Y_1& Y_2 &Y_3 & X_3 & X_2 &X_1 \\ \hline
		Y_1 &&&&0&0&-Z \\
		Y_2 &&\textnormal{\Large{0}}&&0&-Z&0 \\
		Y_3 &&&&-Z&-\tfrac{1}{2} Y_1&\tfrac{1}{2} Y_2 \\ \hline
		X_3 &0&0&Z&0&0&0 \\
		X_2 &0&Z&\tfrac{1}{2} Y_1&0&0&-X_3 \\
		X_1 &Z&0&-\tfrac{1}{2} Y_2&0&X_3&0 \\
	\end{array}\right]. \label{LieBracketH21}
	\end{align}
The center of $\HA{1}{2}$ is $\Liez{\HA{1}{2}} = \R Z 
$ and is one-dimensional, the Heisenberg Lie algebra $\mathfrak{h}_1=
\Rspan{ X_3, X_2, X_1}$ is a subalgebra of $\HA{1}{2}$ and its complement
$\pid:= \Rspan{Z, Y_1, Y_2, Y_3}$ is Abelian and  an ideal of $\HA{1}{2}$.
We refer to \cite{Dyn1} and \cite{FollMeta} for the quantization-theoretical intuition behind the definition of the Dynin-Folland group.

Choosing $l = \rp Z^* \in \Liez{\mathfrak{h}_1}^*$ with $\rp \in \R \setminus \{ 0 \}$,  the matrix representation of the symplectic form $B_l$ is given by
	\begin{align} \label{eq:c26}
	[B_l]
	=
	\left[\begin{array}{c|ccc|ccc}
		l([\,.\,, \,.\,])&Y_1& Y_2 &Y_3 & X_3 & X_2 &X_1 \\ \hline
		Y_1 &&&&0&0&-\rp \\
		Y_2 &&\textnormal{\Large{0}}&&0&-\rp&0 \\
		Y_3 &&&&-\rp&0&0 \\ \hline
		X_1 &0&0&\rp&&& \\
		X_2 &0&\rp&0&&\textnormal{\Large{0}}& \\
		X_3 &\rp&0&0&&& \\
	\end{array}\right].
	\end{align}
Consequently, $B_l$ is non-degenerate; its Pfaffian is given $\Pf(l) = \Abs{\rp}^3$ and the
coadjoint orbit $\Orbit_l$ is the $6$-dimensional affine subspace $\rp
Z^* + \Rspan{Y_1^*, \ldots, X_1^*}$. It follows that, up to Plancherel measure zero, all coadjoint orbits are flat and agree precisely with the orbits $\Orbit_{\rp Z^*}, \rp \in \R \setminus \{ 0 \}$.
Since $\pid$ is Abelian and of dimension $4$, it is a polarization for all $l \in \Liez{\HA{1}{2}}^*$; moreover, if $\PID := \exp_{\HG{1}{2}}(\pid)$, then $\PID \rquo \HG{1}{2} \cong \mathbf{H}_1$. Consequently,
 all representations $\pi_{\rp} \in SI/Z(\HG{1}{2})$ can be realized on
 $\L{2}{\mathbf{H}_1}$. In particular, it is convenient to realize the $\pi_{\rp}$ in coordinates arising from the strong Malcev basis $\{
 Z, \ldots, X_1 \}$. The representation corresponding to $l := Z^*$, i.e., for $\rp = 1$, was the object of interest in Dynin's account.
 
Of course, there is a straight-forward generalization to the case $n > 1$ and we discuss it in the following subsection. In particular, we will focus on the natural stratification of $\HA{n}{2}$, which is absent in \cite{Dyn1} and \cite{FollMeta}.

\subsection{A Stratification of $\HA{n}{2}$}

The stratification of $\HA{n}{2}$ possesses a natural homogeneous structure, which will be one of the principal tools for our analysis of the harmonic and anharmonic oscillators on $\H$.

Let us first consider $n=1$. In this case the stratification becomes obvious if we exchange the order of $Y_3$ and $X_3$ for the Lie bracket defined by \eqref{LieBracketH21}. The resulting table of Lie brackets is
	\begin{align*}
	\left[\begin{array}{c|ccc|ccc}
		[\,.\,, \,.\,]&Y_1& Y_2 &X_3 & Y_3 & X_2 &X_1 \\ \hline
		Y_1 &&&&0&0&-Z \\
		Y_2 &&\textnormal{\Large{0}}&&0&-Z&0 \\
		X_1 &&&&Z&0&0 \\ \hline
		Y_3 &0&0&-Z&0&-\tfrac{1}{2} Y_1&\tfrac{1}{2} Y_2 \\
		X_2 &0&Z&0&\tfrac{1}{2} Y_1&0&-X_3 \\
		X_3 &Z&0&0&-\tfrac{1}{2} Y_2&X_3&0 \\
	\end{array}\right]. 
	\end{align*}

	\begin{lem}
The Dynin-Folland Lie algebra $\HA{1}{2}$ admits a stratification $\HA{1}{2} = \Lie{g}_3 \oplus \Lie{g}_2 \oplus \Lie{g}_1$:
        \begin{align*}
        	\Lie{g}_3 := \R Z, \hspace{10pt} \Lie{g}_2 := \Rspan{Y_1, Y_2, X_3}, \hspace{10pt} \Lie{g}_1 := \Rspan{Y_3, X_2, X_1}.
        \end{align*}
	\end{lem}

We now generalize this principle for $n \in \N$ and give a definition of $\HA{n}{2}$ which is equivalent to Dynin's originial definition.

	\begin{dfn} \label{DefDF}
Let $n \in \N$. We define the Dynin-Folland Lie algebra $\HA{n}{2}$ to be $\R^{4n+3}$ equipped with the Lie bracket which for the standard basis $\{ Z, Y_1, \ldots, Y_{2n+1}, X_{2n+1}, \ldots, X_1 \}$ satisfies
	\begin{align*}
	\left[\begin{array}{c|ccccccc|ccccccc}
		[\,.\,, \,.\,]&Y_1& \cdots & Y_n & Y_{n+1} & \cdots & Y_{2n} & X_{2n+1} & Y_{2n+1} & X_{2n} & \cdots & X_{n+1} & X_n & \cdots & X_1 \\ \hline
		Y_1 &&&&&&&&&&&&&&-Z \\
		\vdots  &&&&&&&&&&&&&\iddots& \\
		Y_n &&&&&&&&&&&&-Z&& \\
		Y_{n+1} &&&&\textnormal{\Large{0}}&&&&&&&-Z&&& \\
		\vdots &&&&&&&&&&\iddots&&&& \\
		Y_{2n} &&&&&&&&&-Z&&&&& \\
		X_{2n+1} &&&&&&&&Z&&&&&& \\ \hline
		Y_{2n+1} &&&&&&&-Z&&-\frac{1}{2} Y_n& \cdots & -\frac{1}{2} Y_1 & \frac{1}{2} Y_{2n} & \cdots & \frac{1}{2} Y_{n+1}\\
		X_{2n} &&&&&&Z&&\frac{1}{2} Y_n&&&&-X_{2n+1}&& \\
		\vdots &&&&&\iddots&&&\vdots&&&&&\ddots& \\
		X_{n+1} &&&&Z&&&&\frac{1}{2} Y_1&&&&&&-X_{2n+1} \\
		X_n &&&Z&&&&&-\frac{1}{2} Y_{2n}&X_{2n+1}&&&&& \\
		\vdots &&\iddots&&&&&&\vdots&&\ddots&&&& \\
		X_1 &Z&&&&&&&-\frac{1}{2} Y_{n+1}&&&X_{2n+1}&&& \\
	\end{array}\right]. 
	\end{align*}

We define the Dynin-Folland group $\HG{1}{2}$ to be the connected, simply connected nilpotent Lie group obtained by exponentiating $\HA{n}{2}$.
	\end{dfn}

	\begin{rem}
The subalgebra $\Rspan{X_{2n+1}, X_{2n}, \ldots, X_1} \subgr \HA{n}{2}$ is isomorphic to the Heisenberg Lie algebra $\h$.
	\end{rem}

The group $\HG{1}{2}$ admits a stratification and, by Definition~\ref{CanDilGrGr}, a canonical homogeneous structure:

	\begin{lem} \label{LemStrat}
The Dynin-Folland Lie algebra $\HA{n}{2}$ admits a stratification $\HA{n}{2} = \Lie{g}_3 \oplus \Lie{g}_2 \oplus \Lie{g}_1$:
        \begin{align*}
        	\Lie{g}_3 := \R Z, \hspace{10pt} \Lie{g}_2 := \Rspan{Y_1, \ldots, Y_{2n}, X_{2n+1}}, \hspace{10pt} \Lie{g}_1 := \Rspan{Y_{2n+1}, X_{2n}, \ldots, X_1}.
        \end{align*}
	\end{lem}

	\begin{prop} \label{NatDilDF}
The canonical family of dilations $\{ D_r \}_{r > 0}$ on the Dynin-Folland group $\HG{n}{2}$ related to the stratification in Lemma~\ref{LemStrat} is given by
	\begin{equation}
	\left\{\begin{array}{rclcl}
		D_r(Z) &=& r^3 Z, && \\
		D_r(Y_j) &=& r^2 Y_j, \hspace{5pt} D_r(X_{2n+1}) &=& r^2 X_{2n+1}, \\
		D_r(X_j) &=& r X_j, \hspace{5pt} D_r(Y_{2n+1}) &=& r Y_{2n+1},
	\end{array}\right. \label{DilationsDF}
	\end{equation}
for $j = 1, \ldots, 2n$.

For the homogeneous structure thus defined the homogeneous dimension equals 
	\begin{align*}
		\hdim = \sum_{j = 1}^3 j \hspace{1pt} \dim(\Lie{g}_j) = (2n + 1) + 2 (2n+1) + 3 = 6n+6.
	\end{align*}
	\end{prop}

In Section~\ref{SpecEst} we will also consider other dilations which are compatible with this stratification of $\HA{n}{2}$.

\subsection{The Generic Representations} \label{GenReps}

In this subsection we discuss the generic unirreps of $\HG{n}{2}$ in terms of a convenient set of coordinates. The generic unirreps were first classified in \cite{FollMeta}. For $\HA{1}{2}$ these are precisely the representations corresponding to the coadjoint orbits $\Orbit_{\rp Z^*}, \rp \in \R \setminus \{ 0 \}$, see Subsection~\ref{Dynin'sLieAlg}. A classification of the remaining unirreps was given in \cite{FiRoRu}. A consequence of the full classification is the square integrability modulo the respective projective kernels of all unirreps of $\HG{n}{2}$. For details we refer to \cite[p.~7]{FollMeta} and \cite[Cor.~4.7]{FiRoRu}.

In this article, however, we will exclusively use the generic unirreps, of which we give an explicit description. For general $n \in \N$ the generic coadjoint orbits are the affine subspaces
	\begin{align*}
		\Orbit_{\rp Z^*} = \rp Z^* \oplus \R^{4n+2} \subgr \HA{n}{2}^*
	\end{align*}
for $\rp \in \R \setminus \{ 0 \}$ and the corresponding representations denoted by $\pi_{\rp}$, can be realized in the representation space $\L{2}{\H}$: The subalgebra $\pid := \Rspan{Z, Y_1, \ldots, Y_{2n+1}}$ is a polarization for each representative $\rp Z^*, \rp \neq 0$. In fact, $\pid$ is Abelian, thus an ideal of $\HA{n}{2}$; equivalently, $\PID := \exp_{\HG{n}{2}}(\pid)$ is an Abelian normal subgroup of $\HG{n}{2}$. The direct complement $\qa := \pid^\perp = \Rspan{X_{2n+1}, \ldots, X_1}$ is therefore a subalgebra of $\HA{n}{2}$ which is isomorphic to $\h$; equivalently, $\H \cong \exp_{\HG{n}{2}}(\qa) =: \QA \subgr \HG{n}{2}$. Given $l = \rp Z^*$ for $\rp \neq 0$, $\pi_{\rp}$ can be realized as the representation $\pi_{\rp} := \indR{\chi_{\rp Z^*}}{\PID}{\HG{n}{2}}$ acting in $\L{2}{\PID \rquo \HG{n}{2}}$ induced by the character
	\begin{align*}
		\chi_{\rp Z^*}: \PID \to \CF: m \mapsto e^{2 \pi i \bracket{\rp Z^*}{\log(m)}}.
	\end{align*}
Since $\PID \rquo \HG{n}{2} \cong \QA$, the natural representation space of $\pi_{\rp}$ can be identified with $\L{2}{\H}$.

The specific choice of coordinates for a concrete realization of $\pi_{\rp}$ in $\L{2}{\H}$ is a matter of taste and purpose. Choosing exponential coordinates and $\rp = 1$, one follows precisely Dynin's original prescription. However, we will deviate from Dynin's choice and employ a different set of coordinates, which simplifies the $\HG{n}{2}$-group law and, a fortiori, the expressions for the generic unirreps. These coordinates are the natural coordinates when $\HG{n}{2}$ is written as the semi-direct product $\HG{n}{2} = \PID \rtimes \QA$ and are referred to as ``split exponential coordinates'' in \cite{FiRoRu}.

Writing an arbitrary, but fixed element $g \in \HG{n}{2}$ in split exponential coordinates amounts to writing
	\begin{align*}
		g = \exp_{\HG{n}{2}} \bigl( zZ + y_1 Y_1 + \ldots + y_{2n+1} Y_{2n+1} \bigr)  \exp_{\HG{n}{2}} \bigl( x_{2n+1} X_{2n+1} + \ldots +x_1 X_1 \bigr)
	\end{align*}		
for uniquely determined $z, y_1, \ldots, x_1 \in \R$. 
We will denote by $y$ and $x$ the coordinate vectors $(y_1, \ldots, y_{n+1})$ and $(x_{n+1}, \ldots, x_1)$, respectively, in order to set the notation
	\begin{align}
		\exp_{\HG{n}{2}} \bigl( zZ + y_1 Y_1 + \ldots + y_{2n+1} Y_{2n+1} \bigr)  \exp_{\HG{n}{2}} \bigl( x_{2n+1}& X_{2n+1} + \ldots +x_1 X_1 \bigr) \nn \\
		 &:= \bigl(z, y_1, \ldots, y_{n+1}, x_{n+1}, \ldots, x_1 \bigr) := (z, y, x). \label{spexcs}
	\end{align}		
For $g = (z, y, x)$ and $g' = (z', y', x')$ the $\HG{n}{2}$-group multiplication is then given by
	\begin{align*}
		(z, y, x) (z', y', x') = (z'', y'', x'') 
	\end{align*}
with
	\begin{equation} \nn 
	\left\{ \begin{array}{rcl}
		z'' &=& z + z' + \sum_{j = 1}^{2n+1} x_j y'_j = z + z' + \bracket{x}{y'},  \\
		y''_1 &=& y_1 + y'_1 + \frac{1}{2}  y'_{2n+1} x_{n+1}, \\
		& \vdots & \\
		y''_n &=& y_n + y'_n + \frac{1}{2} y'_{2n+1} x_{2n}, \\
		y''_{n+1}&=& y_{n+1} + y'_{n+1} - \frac{1}{2} y'_{2n+1} x_1, \\
		& \vdots & \\
		y''_{2n}&=& y_{2n} + y'_{2n} - \frac{1}{2} y'_{2n+1} x_n, \\
		y''_{2n+1} &=& y_{2n+1} + y'_{2n+1}, \\
		x''_{2n+1} &=& x_{2n+1} + x'_{2n+1} + \frac{1}{2} \sum_{j=1}^n (x_j x'_{n+j} - x'_j x_{n+j}), \\
		x''_{2n} &=& x_{2n} + x'_{2n}, \\
		& \vdots & \\
		x''_1 &=& x_1 + x'_1. \\
	\end{array}\right. \nn
	\end{equation}
If we denote by $x \cdot x'$ the $x''$-coordinates of $(0, 0, x'') = (0, 0, x) \cdot (0, 0, x') \in \QA \cong \H$ and by $\coad$ the coadjoint action of $\H$ on $\h^* \cong \R^{2n+1}$, we can abbreviate the group law substantially by
	\begin{align}
		(z, y, x) (z', y', x') = \bigl( z + z' + \bracket{x}{y'}, y + y' + \frac{1}{2} \coad(x)y', x \cdot x' \bigr). \label{GrLawAbbr}
	\end{align}
For a detailed computation we refer to \cite{FollMeta} and \cite[\SS~3.2]{FiRoRu}. A straightforward computation yields the action of $\pi_{\rp}$ on $f \in \L{2}{\H}$ in terms of split exponential coordinates:
	\begin{align}
		\bigl( \pi_{\rp}(z, y, x) f \bigr)(t) = e^{2 \pi i \rp z} \hspace{2pt} e^{2 \pi i \rp \bracket{t}{y}} \hspace{2pt} f(t \cdot x) \label{GenGrRep}
	\end{align}
for $t =(t_{2n+1}, \ldots, t_1) := \exp_{\H} \bigl( t_{2n+1} X_{2n+1} + \ldots + t_1 X_1 \bigr) \in \H$. We note that the restriction of $\pi_{\rp}$ to the subgroup $\QA \cong \H$ equals the right regular representation $R_{\H}$. For more details we refer to \cite{FollMeta} and \cite[\SS~4.2]{FiRoRu}.

This realization will be our default realization of $\pi_{\rp}$.

	\begin{dfn} \label{DefRep}
Let the group $\HG{n}{2}$ defined by Definition~\ref{DefDF} be equipped with the split exponential coordinates given by \eqref{spexcs}. For the unitary irreducible representation $\pi_{\rp}$ of $\HG{n}{2}$ corresponding to the coadjoint orbit with representative $l \in \Lie{\HA{n}{2}} \setminus \{ 0 \}$ we define the ``realization of $\pi_{\rp}$ in split exponential coordinates'' to be the one in the representation space $\HS_{\pi_{\rp}} = \L{2}{\H}$ given by \eqref{GenGrRep}.

For the representative $Z^*$, i.e., $\rp = 1$, we abbreviate $\pi_1$ by $\pi$.
	\end{dfn}

\section{The Harmonic Oscillator on $\H$ as a Sum of Squares} \label{SectionHHO}

In this section we propose a definition of ``harmonic oscillator on $\H$'' which seeks to generalize the geometric aspects of the identity
	\begin{align*}
		\HO = - d\rho \bigl( \SL_{\H}) = - \Delta + 4 \pi^2 \hspace{1pt} \Abs{t}^2
	\end{align*}
for the harmonic oscillator on $\R^n$. The definition will be a direct consequence of the representation-theoretic results discussed in the previous section and in this section.

The representation $d\pi_{\rp}$ of the Lie algebra $\HA{n}{2}$ in a vector space of differential operators on $\H$ was Dynin's starting point for defining $\HA{n}{2}$ and, a fortiori, the Lie group $\HG{n}{2}$. Conversely, one obtains the Lie algebra representation from the Lie group representation by differentiation. Having fixed the realization of $\pi_{\rp}$ in the representation space $\HS_{\pi_{\rp}} = \L{2}{\H}$, the image $d\pi_{\rp}(\HA{n}{2})$ is the isomorphic $4n+3$-dimensional Lie algebra of linear operators whose natural domains include the smooth vectors $\HS_{\pi_{\rp}}^\infty = \SFG{\H}$; this Lie algebra is naturally equipped with the commutator bracket and it is determined by the strong Malcev basis $\{ d\pi_{\rp}(Z), d\pi_{\rp}(Y_1), \ldots, d\pi_{\rp}(X_1) \}$ acting on $f \in \SFG{\H}$ by
	\begin{equation} \nn 
	\left\{ \begin{array}{rcl}
		d\pi_{\rp}(Z) f  &=& 2 \pi i \rp \hspace{2pt} f, \\
		d\pi_{\rp}(Y_j) f &=& 2 \pi i  \rp \hspace{2pt}  t_j \hspace{1pt} f, \hspace{80pt} j = 1, \ldots, 2n+1, \\
		d\pi_{\rp}(X_{2n+1}) f &=& \partial_{t_{2n+1}} f, \\
		d\pi_{\rp}(X_{2n}) f &=& \bigl( \partial_{t_{2n}} + \frac{1}{2} t_n \partial_{t_{2n+1}}\bigr) f, \\
		& \vdots & \\
		d\pi_{\rp}(X_{n+1}) f &=& \bigl( \partial_{t_{n+1}} + \frac{1}{2} t_1 \partial_{t_{2n+1}} \bigr) f, \\
		d\pi_{\rp}(X_n) f  &=& \bigl( \partial_{t_n} - \frac{1}{2} t_{2n} \partial_{t_{2n+1}} \bigr) f, \\
		& \vdots & \\
		d\pi_{\rp}(X_1) f  &=& \bigl( \partial_{t_1} - \frac{1}{2} t_{n+1} \partial_{t_{2n+1}} \bigr) f.
	\end{array}\right. \nn
	\end{equation}
Not surprisingly, the differential operators $d\pi_{\rp}(X_{2n+1}), \ldots, d\pi_{\rp}(X_1)$ are precisely the left-invariant vector fields $dR_{\H}(X_{2n+1}), \ldots, dR_{\H}(X_1)$ on $\H$ corresponding to the strong Malcev basis of $\h$. Consequently, we have the following.

	\begin{lem}
For each $\rp \in \R \setminus \{ 0 \}$ let us denote by $d\pi_{\rp}$ also the extension of the representation $d\pi_{\rp}$ of $\HA{n}{2}$ to the universal enveloping algebra $\UEA{\h}$. For $\pi = \pi_1$ we have
	\begin{align*}
		\SL_{\H} = d\pi(X_1)^2 + \ldots + d\pi(X_{2n})^2 = d\pi \bigl( X_1^2 + \ldots + X_{2n}^2 \bigr),
	\end{align*}
that is, the left sub-Laplacian $\SL_{\H}$ is the image under $d\pi$ of the sum of squares of the basis vectors $X_1, \ldots, X_{2n}$ of $\HA{n}{2}$ which span the subalgebra $\qa \cong \h$.
	\end{lem}

In order to define our version of harmonic oscillator on $\H$ we go one step ahead and consider the sum of squares of the basis vectors spanning the first stratum $\Lie{g}_1 \subgr \HA{n}{2}$.

	\begin{dfn}[Harmonic Oscillator on $\H$] \label{DefHHO}
Let $\{ Z, Y_1, \ldots, Y_{2n+1}, X_{2n+1}, \ldots, X_1 \}$ be the strong Malcev basis of the Dynin-Folland Lie algebra $\HA{n}{2}$ determined by Definition~\ref{DefDF}. For $\rp \in \R \setminus \{ 0 \}$ let $\pi_{\rp}$ be the representation of the Dynin-Folland group $\HG{n}{2}$ in $\L{2}{\H}$ defined by Definition~\ref{DefRep} and denote by $d\pi_{\rp}$ also the extension to the universal enveloping algebra of the Lie algebra representation $d\pi_{\rp}$ of $\HA{n}{2}$. We then define the harmonic oscillator on the Heisenberg group $\H$ of parameter $\rp$ to be the positive essentially self-adjoint operator
	\begin{align*}
		\HHO^\rp :=& - d\pi_{\rp} \bigl( X_1^2 + \ldots X_{2n}^2 + Y_{2n+1}^2 \bigr) \\
		=& - d\pi_{\rp}(X_1)^2 - \ldots  - d\pi_{\rp}(X_{2n})^2 - d\pi_{\rp}(Y_{2n+1})^2  \\ 
		=& - \SL_{\H} + 4 \pi^2 \rp^2 \hspace{1pt} t_{2n+1}^2 \\
		=& - \bigl( \partial_{t_1} - \frac{1}{2} t_{n+1} \partial_{t_{2n+1}} \bigr)^2 - \ldots - \bigl( \partial_{t_{2n}} + \frac{1}{2} t_n \partial_{t_{2n+1}}\bigr)^2 + 4 \pi^2 \rp^2 \hspace{1pt} t_{2n+1}^2,
	\end{align*}
whose natural domain includes the smooth vectors $\HS^\infty_{\pi_{\rp}} \cong \SFG{\H}$.

That is, $\HHO^\rp$ is defined to be the image under $d\pi_{\rp}$ of the negative sum of squares of the basis vectors $Y_{2n+1}, X_{2n}, \ldots, X_1$ which span the first stratum $\Lie{g}_1$ of the stratification $\HA{n}{2} = \Lie{g}_3 \oplus \Lie{g}_2 \oplus \Lie{g}_1$ defined by Defintion~\ref{LemStrat}. For $\rp = 1$ we set $\HHO^1 =: \HHO$.
	\end{dfn}

	\begin{ex}\label{EX:haH1}
In the case $n = 1$ the harmonic oscillator $\mathcal{Q}_{\mathbf{H}_1}^\rp$ has the explicit form
	\begin{align*}
		\mathcal{Q}_{\mathbf{H}_1}^\rp = - \bigl( \partial_{t_1}^2 + \partial_{t_2}^2 \bigr) - \frac{1}{4} \bigl ({t_1}^2 + {t_2}^2 \bigr ) \hspace{1pt} \partial_{t_3}^2 + \bigl ( t_1 \hspace{1pt} \partial_{t_2} - t_2 \hspace{1pt} \partial_{t_1} \bigr ) \hspace{1pt} \partial_{t_3} + 4 \pi^2 \rp^2 \hspace{1pt} t_3^2.
	\end{align*}
	\end{ex}

\section{Spectral Estimates for Harmonic and Anharmonic Oscillators on $\H$} \label{SpecEst}

In this section we apply the results of ter Elst and Robinson~\cite{tERo} to give asymptotic estimates for the distribution of eigenvalues of the harmonic and the related generic anharmonic oscillators on $\H$. Our arguments will roughly follow the order in which we recalled crucial definitions and theorems in Section~\ref{MachinerytERob}.

\subsection{The Harmonic Oscillator-Case}
In this subsection we treat the case of the harmonic oscillator $\HHO^\rp$. To apply the methods from \cite{tERo}, we employ a homogeneous quasi-norm on $\HA{n}{2}^*$ which is both compatible with the dilations from Proposition~\ref{NatDilDF} and convenient to give spectral estimates for the harmonic oscillator $\HHO^\rp$.

	\begin{prop} \label{QNDF}
Let $\{ D_r \}_{r > 0}$ be the family of dilations on the Dynin-Folland Lie algebra $\HA{n}{2}$ defined in Proposition~\ref{NatDilDF} and let $\{ D^*_r \}_{r > 0}$ be the corresponding dilations on $\HA{n}{2}^*$ defined by duality. Moreover, let us express all elements $l \in \HA{n}{2}^*$ in terms of the dual basis of the strong Malcev basis $\{ Z, Y_1, \ldots, X_1 \}$, i.e., $l = l_z Z^* + l_{y_1} Y^*_1 + \ldots + l_{x_1} X^*_1$. Then the map $\qn{\HA{n}{2}^*}{\, . \,}^\infty: \HA{n}{2}^* \to [0, \infty)$ defined by
	\begin{align*}
		\qn{\HA{n}{2}^*}{l}^\infty := \max \bigl \{ \Abs{l_z}^{1/3}, \Abs{l_{y_1}}^{1/2}, \ldots, \Abs{l_{y_{2n}}}^{1/2}, \Abs{l_{x_{2n+1}}}^{1/2}, \Abs{l_{y_{2n+1}}}, \Abs{l_{x_{2n}}}, \ldots, \Abs{l_{x_1}}  \bigr \}
	\end{align*}
is a homogeneous quasi-norm.
	\end{prop}

	\begin{proof}
This is a special case of Corollary~\ref{CQNGrGr}.
	\end{proof}

The following statement about the harmonic oscillator $\HHO$ follows automatically from a general statement about sub-Laplacians on stratified Lie groups: The negative (left or right) sub-Laplacian $- \SL_G$ on a stratified Lie group $G$ is a (left-invariant or right-invariant) positive Rockland operator. For details we refer to \cite{FiRuMon}, Lemma~4.1.7 and Remark~4.2.4. Since every left or right sub-Laplacian equals $dR(X_1^2+ \ldots + X_d^2)$ or $dL(X_1^2+ \ldots + X_d^2)$, respectively, for a basis $X_1, \ldots, X_d$ of the first stratum of $\Lie{g}$, we have the following.

	\begin{lem}
The element $P = X_1^2 + \ldots X_{2n}^2 + Y_{2n+1}^2 \in \UEA{\HA{n}{2}}$, employed in Definition~\ref{DefHHO}, is a positive Rockland form of homogeneous degree $\hdeg = 2$.
	\end{lem}

	\begin{prop} \label{SpecEstHHO}
 Let $\pi_{\rp}$ be the unitary irreducible representation of the Dynin-Folland group $\HG{n}{2}$ corresponding to the representative $\rp Z^*, \rp \in \R \setminus \{ 0 \}$. Let $P = X_1^2 + \ldots + X_{2n}^2 + Y_{2n+1}^2 \in \UEA{\HA{n}{2}}$ and denote by $\hdeg$ the homogeneous degree of $P$, which equals $2$. Denote by $\hdim_{\Lie{z}}$ the homogeneous dimension of the center $\Liez{\HA{n}{2}}$, which equals $3$. Then the number of eigenvalues of
the harmonic oscillator $\HHO^\rp = d\pi_{\rp}(P)$ on $\H$, counted with multiplicities, which are less or equal $\l > 0$ is asymptotically given by
	\begin{align*}
		N(\l, \pi_{\rp}, \RF) \asymp \Abs{\rp}^{-(2n+1)} \hspace{1pt} \l^{\frac{\hdim - \hdim_{\Lie{z}}}{\hdeg}} = \Abs{\rp}^{-(2n+1)} \hspace{1pt} \l^{\frac{6n + 3}{2}}
	\end{align*}
and
	\begin{align*}
		\l_{\rp, s} \asymp \Abs{\rp}^{\frac{2}{3}} \hspace{1pt} s^{ \frac{2}{6n+3}} \hspace{5pt} \mbox{ for } \hspace{5pt} s = 1, 2, \ldots,
	\end{align*}
with equivalence constants which are uniform in $\rp \in \R \setminus \{ 0 \}$, and $\l$ and $\l_{\rp, s}$, respectively.
	\end{prop}

	\begin{proof}
To begin with, we observe that on each coadjoint orbit $\Orbit_{\rp Z^*}, \rp \in \R \setminus \{ 0 \},$ the orbital measure is given by $\om_{\pi_{\rp}} = \Abs{\rp}^{-(2n+1)} \hspace{1pt} \Leb_{\R^{4n+2}}$, that is, a weighted version of the $(4n+2)$-dimensional Lebesgue measure on the affine subspace $\Orbit_{\rp Z^*} \subseteq \HA{n}{2}^*$. This follows directly from the direct integral representation $\Leb_{\HA{n}{2}^*} = \int^\oplus_{\Liez{\HA{n}{2}}^*} \om_{\pi_{\rp}} d\Pla(\pi_{\rp})$ of the Lebesgue measure on $\HA{n}{2}^*$ by the Plancherel theorem (cf.~\cite[Thm.~4.3.10]{CoGr}) and the fact that the Plancherel measure on $\Liez{\HA{n}{2}}^* \cong \HA{n}{2}^*/\coAd(\HG{n}{2}) \cong \widehat{\mathbf{H}}_{n, 2}$ is given by $d\Pla(\pi_{\rp}) = \Abs{\rp}^{2n+1} \hspace{1pt} d\rp$ (cf.~\cite[p.~9]{FollMeta} or \cite[Thm.~3.18]{Ro14}). 

To compute the asymptotic number of eigenvalues, we employ the homogeneous quasi-norm $\qn{\HA{n}{2}^*}{\, . \,}^\infty$ from Proposition~\ref{QNDF}. Since for $\l < \Abs{\rp}^{1/3}$ the $\qn{\HA{n}{2}^*}{ \, . \,}^\infty$-ball of radius $\l$ around $0$ does not intersect the orbit $\Orbit_{\rp Z^*}$, we may assume $\l \geq \Abs{\rp}^{1/3}$ and compute
	\begin{align*}
		N_0(\l, \pi_{\rp}) &= \om_{\pi_{\rp}} \bigl( \bigl \{ l \in \Orbit_{\rp Z^*} : \qn{\HA{n}{2}^*}{l}^\infty \leq \l \bigr \} \bigr) \\
		 &= \Abs{\rp}^{-(2n+1)} \hspace{1pt} \Leb_{\R^{4n+2}} \bigl( \bigl \{ l \in \Orbit_{\rp Z^*} : \Abs{l_{y_1}} \leq \l^2, \ldots, \Abs{l_{x_1}} \leq \l \bigr \} \bigr ) \\
		 &= \Abs{\rp}^{-(2n+1)} \hspace{1pt} \l^{2(2n+1) + 2n+1} \\
		 &= \Abs{\rp}^{-(2n+1)} \hspace{1pt} \l^{6n+3}.
	\end{align*}
Hence, by \cite[Thm.~4.1]{tERo} we have $N_0(\l^{1/2}, \pi_{\rp}) \asymp N(\l, \pi_{\rp}, \RF) \asymp \Abs{\rp}^{-(2n+1)} \hspace{1pt} \l^{ \frac{6n+3}{2}}$.

The second estimate follows from the observation that for small enough $\e > 0$ and large enough $s_0 \in \N$ there exists a constant $c = c(\e, s_0) > 0$ independent of $\rp$ such that for all $s \geq s_0$
	\begin{align*}
		c^{-1} \Abs{\rp}^{-(2n+1)}(\l_{\rp, s} - \e)^{\frac{6n+3}{2}} \leq N(\l_s - \e, \pi_{\rp}, \RF) \leq s \leq N(\l_s, \pi_{\rp}, \RF) \leq c \Abs{\rp}^{-(2n+1)} \l_{\rp, s}^{\frac{6n+3}{2}},
	\end{align*}
or equivalently, $\l_{\rp, s} \asymp \Abs{\rp}^{\frac{2}{3}} \hspace{1pt} s^{ \frac{2}{6n+3}}$.
	\end{proof}

\subsection{The Anharmonic Oscillator-Case} \label{Subsection_HAO}

In this subsection we study a class of operators which we call the anharmonic oscillators on $\H$. These operators share one main feature: Every gradable Lie algebra equipped with a specific gradation admits a non-empty set of positive Rockland forms. Now we assign to each positive Rockland form $\RF \in \UEA{\HA{n}{2}}$ a family of anharmonic oscillators $\{ \HAO^\rp \}_{\rp \in \R \setminus \{ 0 \}}$; the family of harmonic oscillators $\{ \HHO^\rp \}_{\rp \in \R \setminus \{ 0 \}}$ is included as the special case in which $\HA{n}{2}$ is equipped with the canonical homogeneous structure related to the natural stratification and $\RF = - (X_1^2 + \ldots X_{2n}^2 + Y_{2n+1}^2)$.

	\begin{dfn}[Anharmonic Oscillators on $\H$] \label{DefHAO}
Let $\pi_{\rp}$ be the unitary irreducible representation of the Dynin-Folland group $\HG{n}{2}$ corresponding to the representative $\rp Z^*$ with $\rp \in \R \setminus \{ 0 \}$. Let $P \in \UEA{\HA{n}{2}}$ be a positive homogeneous form which is Rockland for a given homogeneous structure on $\HA{n}{2}$. We define the anharmonic oscillator on $\H$ associated to $\RF$ and $\rp$ as the operator
	\begin{align*}
		\HAO^\rp := d\pi_{\rp}(P).
	\end{align*}
	\end{dfn}

We start our analysis with the following proposition, which characterizes a large class of homogeneous structures on $\HA{n}{2}$.

	\begin{prop} \label{AnhDilDF}
Let $\vt_1, \ldots, \vt_{2n+1} \in \N$. Then
	\begin{equation} \label{GenDilDF}
	\left\{\begin{array}{rclrcl}
		D_r(Z) &:=& r^{\vt_j + \vt_{n+j} + \vt_{2n+1}} Z, &&& \\
		D_r(Y_{n + j}) &:=& r^{\vt_j + \vt_{2n+1}} Y_{n + j}, &D_r(Y_j) &:=& r^{\vt_{n+j} + \vt_{2n+1}} Y_j, \\
		 D_r(X_{2n+1}) &:=& r^{\vt_j + \vt_{n+j}} X_{2n+1}, &D_r(Y_{2n+1}) &:=& r^{\vt_{2n+1}} Y_{2n+1}, \\
		D_r(X_j) &:=& r^{\vt_j} X_j, &D_r(X_{n + j}) &:=& r^{\vt_{n + j}} X_{n + j}
	\end{array}\right. \nn 
	\end{equation}
for $j = 1, \ldots, n,$ defines a family of dilations $\{ D_r \}_{r > 0}$ on the Dynin-Folland Lie algebra $\HA{n}{2}$ if
	\begin{align*}
		\vt_1 + \vt_{n+1} = \ldots = \vt_n + \vt_{2n}.
	\end{align*}

For the homogeneous structure thus defined the homogeneous dimension equals 
	\begin{align*}
		\hdim = (2n+2) (\vt_1 + \vt_{n+1} + \vt_{2n+1}). 
	\end{align*}

In particular, this holds true for $\vt_j = \vt_1, \vt_{n+j} = \vt_{n+1}$ for $j = 1, \ldots, n$.

The countable family of all such homogeneous structures exhausts all homogeneous structures on $\HA{n}{2}$ for which the strong Malcev basis $\{ Z, Y_1, \ldots, X_1 \}$ forms an eigenbasis of the matrix $A$ in $D_r := \exp_{GL(n, \R)} \bigl( A \log(r) \bigr)$. Moreover, two such homogeneous structures with weights $\vt_1, \ldots, \vt_{2n+1} \in \N$ and $\vt'_1, \ldots, \vt'_{2n+1} \in \N$ coincide if there exits an integer $r_0 \in \N$ such that $\vt_j = r_0 \hspace{1pt} \vt'_j$ for all $j = 1, \ldots, 2n+1$.
	\end{prop}

	\begin{proof}
The calculations are easy and straight-forward, so we omit them. In order to check that these are the only admissible homogeneous structures with the specific eigenbasis property, we note that there are no other possible combinations of weights $\vt_1, \ldots, \vt_{2n+1} \in \N$ which are compatible with the Lie bracket on $\HA{n}{2}$. Since for graded groups the weights can always be scaled so that all weights are positive integers and their greatest common divisor equals $1$, this proves the claim.
	\end{proof}

For $\vt_1 = \ldots = \vt_{2n+1} = 1$ we obtain the canonical homogeneous structure from Proposition~\ref{NatDilDF}.

As for the harmonic oscillators $\HHO^\rp$, the key to the spectral estimates in \cite{tERo} lies in the choice of a convenient homogeneous quasi-norm on $\HA{n}{2}$.  The following proposition is a special case of Proposition~\ref{QNGrGr}, in which the dilation matrix $A$ is diagonalized by the vectors of the strong Malcev basis.

	\begin{prop} \label{AnhQNDF}
Let $\{ D_r \}_{r > 0}$ be the family of dilations defined in Proposition~\ref{AnhDilDF} for fixed weights $\vt_1, \ldots, \vt_{2n+1} \in \N$, and let $\{ D^*_r \}_{r > 0}$ be the corresponding dilations on $\HA{n}{2}^*$ defined by duality. Moreover, let us express all elements $l \in \HA{n}{2}^*$ in terms of the dual basis of the strong Malcev basis $\{ Z, Y_1, \ldots, X_1 \}$, i.e., $l = l_z Z^* + l_{y_1} Y^*_1 + \ldots + l_{x_1} X^*_1$ for some $l_z, \ldots, l_{x_1} \in \R$. Then the map $\qn{\HA{n}{2}^*, A}{\, . \,}^{\infty}: \HA{n}{2}^* \to [0, \infty)$ defined by
	\begin{align*}
		\qn{\HA{n}{2}^*, A}{l}^{\infty} := \max \Bigl \{ &\Abs{l_z}^{1/(\vt_1 + \vt_{n+1} + \vt_{2n+1})}, \Abs{l_{y_1}}^{1/(\vt_{n+1} + \vt_{2n+1})}, \ldots, \Abs{l_{y_n}}^{1/(\vt_{2n} + \vt_{2n+1})}, \\
		&\Abs{l_{y_{n+1}}}^{1/(\vt_1 + \vt_{2n+1})}, \ldots, \Abs{l_{y_{2n}}}^{1/(\vt_n + \vt_{2n+1})}, \Abs{l_{x_{2n+1}}}^{1/(\vt_1 + \vt_{n+1})}, \\
		&\Abs{l_{y_{2n+1}}}^{1/\vt_{2n+1}}, \Abs{l_{x_{2n}}}^{1/\vt_{2n}}, \ldots, \Abs{l_{x_{n+1}}}^{1/\vt_{n+1}}, \Abs{l_{x_n}}^{1/\vt_n}, \ldots, \Abs{l_{x_1}}^{1/\vt_1}  \Bigr \}
	\end{align*}
is a homogeneous quasi-norm.
	\end{prop}

We can now state the main result of this section.

	\begin{prop} \label{SpecEstHAO}
Let $\pi_{\rp}$ be the unitary irreducible representation of the Dynin-Folland group $\HG{n}{2}$ corresponding to the representative $\rp Z^*, \rp \in \R \setminus \{ 0 \}$. Given a family of dilations on $\HA{n}{2}$ of the type \eqref{GenDilDF}, with weights $\vt_1, \ldots, \vt_{2n+1} \in \N$, and a positive homogeneous Rockland form $\RF \in \UEA{\HA{n}{2}}$, let $\HAO^\rp$ be the anharmonic oscillator on $\H$ associated to $P$ and $\rp$, defined by Definition~\ref{DefHAO}.  Denote by $\hdim_{\Lie{z}}$ the homogeneous dimension of the center $\Liez{\HA{n}{2}}$, which equals $\vt_1 + \vt_{n+1} + \vt_{2n+1}$, and denote by $\hdeg$ the homogeneous degree of $P$ and, a fortiori, of $\HAO^\rp$. Then the number of eigenvalues of $\HAO^\rp$, counted with multiplicities, which are less or equal $\l > 0$ is asymptotically given by
	\begin{align*}
		N(\l, \pi_{\rp}, \RF) \asymp \Abs{\rp}^{-(2n+1)} \hspace{1pt} \l^{\frac{\hdim - \hdim_{\Lie{z}}}{\hdeg}}
		= \Abs{\rp}^{-(2n+1)} \hspace{1pt} \l^{\frac{ (2n+1) (\vt_1 + \vt_{n+1} + \vt_{2n+1}) }{\hdeg}}
	\end{align*}
and
	\begin{align*}
		\hspace{20pt} \l_{\rp, s}
		\asymp \bigl( \Abs{\rp}^{2n+1} \hspace{1pt} s \bigr )^{\frac{\hdeg}{\hdim - \hdim_{\Lie{z}}}}
		\hspace{10pt} \mbox{ for } \hspace{10pt} s = 1, 2, \ldots,
	\end{align*}
with constants which are uniform in $\rp \in \R \setminus \{ 0 \}$, and $\l$ and $\l_{\rp, s}$, respectively.
	\end{prop}

	\begin{proof}
The proof is almost identical to the proof of Proposition~\ref{SpecEstHHO}; for this reason we focus on the only crucial difference, the measure $N_0(\l, \pi_{\rp})$ for the quasi-norm $\qn{\HA{n}{2}^*, A}{\, . \,}^{\infty}$.
Thus we compute
	\begin{align*}
		N_0(\l, \pi_{\rp}) &= \om_{\pi_{\rp}} \bigl( \bigl \{ l \in \Orbit_{\rp Z^*} : \qn{\HA{n}{2}^*, A}{l}^{\infty} \leq \l \bigr \} \bigr) \\
		 &= \Abs{\rp}^{-(2n+1)} \Leb_{\R^{4n+2}} \Bigl( \bigl \{ l \in \Orbit_{\rp Z^*} : \Abs{l_{y_j}} \leq \l^{\vt_{n+j} + \vt_{2n+1}}, \Abs{l_{y_{n+j}}} \leq \l^{\vt_j + \vt_{2n+1}}, \Abs{l_{x_{2n+1}}} \leq \l^{\vt_1 + \vt_{n+1}}, \\
		 & \hspace{55pt} \Abs{l_{y_{2n+1}}} \leq \l^{\vt_{2n+1}}, \Abs{l_{x_{n+j}}} \leq \l^{\vt_{n+j}}, \Abs{l_{x_j}} \leq \l^{\vt_j}, \hspace{5pt} j = 1, \ldots, n \bigr \} \Bigr ) \\
		 &= \Abs{\rp}^{-(2n+1)} \l^{(2n+1) (\vt_1 + \vt_{n+1} + \vt_{2n+1})}.
	\end{align*}
Hence by \cite[Thm.~4.1]{tERo} we have $N_0(\l^{\frac{1}{\hdeg}}, \pi_{\rp}) \asymp N(\l, \pi_{\rp}, \RF) \asymp \Abs{\rp}^{-(2n+1)} \hspace{1pt} \l^{\frac{(2n+1) (\vt_1 + \vt_{n+1} + \vt_{2n+1})}{\hdeg}}$. 
	\end{proof}

Let us provide a few examples. The first example is generic in the sense that it covers the class of so-called ``classical'' Rockland forms in $\UEA{\HA{n}{2}}$ whose positivity and homogeneity follows from a general statement for graded groups equipped with a gradation with weights $\vt_1, \vt_2, \ldots$. Let us recall this statement; for a proof we refer to \cite[Lem.~4.1.8]{FiRuMon}.

	\begin{lem}
Let $G$ be a graded group of topological dimension $\dimG$ equipped with a family of dilations $\{ D_r \}_{r > 0}$ with weights $\vt_1, \ldots, \vt_\dimG$. If we fix an eigenbasis $\{ X_1, \ldots, X_\dimG \}$ of $A$ in $D_r := \exp_{GL(n, \R)} \bigl( A \log(r) \bigr)$ and if $\hdeg_0$ is a common multiple of $\vt_1, \ldots, \vt_\dimG$, then the operator
	\begin{align}
		\sum_{j=1}^\dimG (-1)^{\frac{\hdeg_0}{\vt_j}} c_j \hspace{1pt} X_j^{\frac{2 \hdeg_0}{\vt_j}} \label{GenClassicalRF}
	\end{align}
with $c_1, \ldots, c_\dimG > 0$ is a positive Rockland operator of homogeneous degree $\hdeg = 2 \hdeg_0$.
	\end{lem}

	\begin{ex} [Generic ``Classical'' Anharmonic Oscillator on $\H$] \label{GenericGradedEx}
Given a homogeneous structure on $\HA{n}{2}$ characterized by Proposition~\ref{AnhDilDF} with weights $\vt_1, \dots, \vt_{2n+1} \in \N$, and given a common multiple $\hdeg_0$ of the weights $\vt_1, \ldots, \vt_{2n+1}, \vt_{1} + \vt_{n+1}, \vt_1 + \vt_{2n+1}, \dots, \vt_{2n} + \vt_{2n+1}, \vt_1 + \vt_{n+1} + \vt_{2n+1}$, the form $P \in \UEA{\HA{n}{2}}$ defined by
	\begin{align*}
		P := \sum_{j=1}^n &(-1)^{\frac{\hdeg_0}{\vt_j}} c_{X_j} \hspace{1pt} X_j^{\frac{2 \hdeg_0}{\vt_j}} + (-1)^{\frac{\hdeg_0}{\vt_{n+j}}} c_{X_{n+j}} \hspace{1pt} X_{n+j}^{\frac{2 \hdeg_0}{\vt_{n+j}}} + (-1)^{\frac{\hdeg_0}{\vt_{2n+1}}} c_{Y_{2n+1}} \hspace{1pt} Y_{2n+1}^{\frac{2 \hdeg_0}{\vt_{2n+1}}} \\
		+ &(-1)^{\frac{\hdeg_0}{\vt_1 + \vt_{n+1}} } c_{X_{2n+1}} \hspace{1pt} X_{2n+1}^{\frac{2 \hdeg_0}{\vt_1 + \vt_{n+1}} }+ (-1)^{\frac{\hdeg_0}{ \vt_{n+j} + \vt_{2n+1}} } c_{Y_j} \hspace{1pt} Y_j^{\frac{2 \hdeg_0}{ \vt_{n+j} + \vt_{2n+1}} } \\
		+ &(-1)^{\frac{\hdeg_0}{\vt_j + \vt_{2n+1}} } c_{Y_{n+j}} \hspace{1pt} Y_{n+j}^{\frac{2 \hdeg_0}{\vt_j + \vt_{2n+1}} }
		+ (-1)^{\frac{\hdeg_0}{\vt_1 + \vt_{n+1} + \vt_{2n+1}} } c_Z \hspace{1pt} Z^{\frac{2 \hdeg_0}{\vt_1 + \vt_{n+1} + \vt_{2n+1}} }
	\end{align*}
with $c_{X_1}, \ldots, c_Z > 0$ is a positive Rockland form of homogeneous degree $\hdeg = 2 \hdeg_0$.  The associated anharmonic oscillator of parameter $\rp \in \R \setminus \{ 0 \}$ is given by
	\begin{align*}
		\HAO^\rp = \sum_{j=1}^n &(-1)^{\frac{\hdeg_0}{\vt_j}} c_{X_j} \hspace{1pt} \bigl( \partial_{t_j} - \frac{1}{2} t_{n+j} \partial_{t_{2n+1}} \bigr)^{2 \frac{\hdeg_0}{\vt_j}} 
		+ (-1)^{\frac{\hdeg_0}{\vt_{n+j}}} c_{X_{n+j}} \hspace{1pt} \bigl( \partial_{t_{n+j}}
		+ \frac{1}{2} t_j \partial_{t_{2n+1}} \bigr)^{2 \frac{\hdeg_0}{\vt_{n+j}}} \\
		+ &(-1)^{\frac{\hdeg_0}{\vt_{2n+1}}} c_{Y_{2n+1}} \hspace{1pt} (2 \pi \Abs{\rp} \hspace{1pt} t_{2n+1})^{\frac{2 \hdeg_0}{\vt_{2n+1}}}
		+ (-1)^{\frac{\hdeg_0}{\vt_1 + \vt_{n+1}} } c_{X_{2n+1}} \hspace{1pt} \partial_{t_{2n+1}}^{\frac{2 \hdeg_0}{\vt_1 + \vt_{n+1}} } \\
		+ &(-1)^{\frac{\hdeg_0}{ \vt_{n+j} + \vt_{2n+1}} } c_{Y_j} \hspace{1pt} (2 \pi \Abs{\rp} \hspace{1pt} t_j)^{\frac{2 \hdeg_0}{ \vt_{n+j} + \vt_{2n+1}} }
		+ (-1)^{\frac{\hdeg_0}{\vt_j + \vt_{2n+1}} } c_{Y_{n+j}} \hspace{1pt} (2 \pi \Abs{\rp} \hspace{1pt} t_{n+j})^{\frac{2 \hdeg_0}{\vt_j + \vt_{2n+1}} }
		 \\
		+ &(-1)^{\frac{\hdeg_0}{\vt_1 + \vt_{n+1} + \vt_{2n+1}} } c_Z \hspace{1pt} (2 \pi \Abs{\rp})^{\frac{2 \hdeg_0}{\vt_1 + \vt_{n+1} + \vt_{2n+1}} }.
	\end{align*}
We then have $N(\l, \pi_{\rp}, \RF) \asymp \Abs{\rp}^{-(2n+1)} \hspace{1pt} \l^{\frac{(2n+1)(\vt_1 + \vt_{n+1} + \vt_{2n+1})}{\hdeg_0}}$.
	\end{ex}

On stratified groups one may restrict the basis vectors in \eqref{GenClassicalRF} to those which span the first stratum. Again, we recall the general result and refer to \cite[Cor.~4.1.10]{FiRuMon} for a proof.

	\begin{lem}
Let $G$ be a stratified group of topological dimension $\dimG$ equipped with a family of dilations $\{ D_r \}_{r > 0}$ with weights $\vt_1, \ldots, \vt_\dimG$. If $\{ X_1, \ldots, X_\dimG \}$ is an eigenbasis of $A$ in $D_r := \exp_{GL(n, \R)} \bigl( A \log(r) \bigr)$ whose first $\dimG'$-many vectors span the first stratum of $\Lie{g}$\footnote{In particular, this holds true for the canonical homogeneous structure related to the stratification and any basis given as the union of bases of the strata.} and if $\hdeg_0$ is a common multiple of $\vt_1, \ldots, \vt_{\dimG'}$, then the operator
	\begin{align}
		\sum_{j=1}^{\dimG'} (-1)^{\frac{\hdeg_0}{\vt_j}} c_j \hspace{1pt} X_j^{\frac{2 \hdeg_0}{\vt_j}} \label{GenClassicalRF}
	\end{align}
with $c_1, \ldots, c_{\dimG'} > 0$ is a positive Rockland operator of homogeneous degree $\hdeg = 2 \hdeg_0$.
	\end{lem}

	\begin{ex} \label{GenericStratEx}
Given a homogeneous structure on $\HA{n}{2}$ characterized by Proposition~\ref{AnhDilDF} with weights $\vt_1, \dots, \vt_{2n+1} \in \N$, and given a common multiple $\hdeg_0$ of these weights, the form $P \in \UEA{\HA{n}{2}}$ defined by
	\begin{align*}
		P := \sum_{j=1}^n \left ( (-1)^{\frac{\hdeg_0}{\vt_j}} c_j \hspace{1pt} X_j^{2 \frac{\hdeg_0}{\vt_j}} + (-1)^{\frac{\hdeg_0}{\vt_{n+j}}} c_{n+j} \hspace{1pt} X_{n+j}^{2 \frac{\hdeg_0}{\vt_{n+j}}} \right ) + (-1)^{\frac{\hdeg_0}{\vt_{2n+1}}} c_{2n+1} \hspace{1pt} Y_{2n+1}^{2 \frac{\hdeg_0}{\vt_{2n+1}}}
	\end{align*}
with $c_1, \ldots, c_{2n+1} > 0$ is a positive Rockland form of homogeneous degree $\hdeg = 2 \hdeg_0$. The associated anharmonic oscillator of parameter $\rp \in \R \setminus \{ 0 \}$ is given by
	\begin{align*}
		\HAO^\rp = \sum_{j=1}^n \Biggl ( &(-1)^{\frac{\hdeg_0}{\vt_j}} c_j \hspace{1pt} \bigl( \partial_{t_j} - \frac{1}{2} t_{n+j} \partial_{t_{2n+1}} \bigr)^{2 \frac{\hdeg_0}{\vt_j}} 
		+ (-1)^{\frac{\hdeg_0}{\vt_{n+j}}} c_{n+j} \hspace{1pt} \bigl( \partial_{t_{n+j}} + \frac{1}{2} t_j \partial_{t_{2n+1}} \bigr)^{2 \frac{\hdeg_0}{\vt_{n+j}}} \Biggr ) \\
		+ &(-1)^{\frac{\hdeg_0}{\vt_{2n+1}}} c_{2n+1} \hspace{1pt} (2 \pi \Abs{\rp} \hspace{1pt} t_{2n+1})^{2 \frac{\hdeg_0}{\vt_{2n+1}}}.
	\end{align*}
We then have $N(\l, \pi_{\rp}, \RF) \asymp \Abs{\rp}^{-(2n+1)} \hspace{1pt} \l^{\frac{(2n+1)(\vt_1 + \vt_{n+1} + \vt_{2n+1})}{\hdeg_0}}$.
	\end{ex}

The following examples will be concrete special cases of the generic Examples~\ref{GenericGradedEx} and \ref{GenericStratEx}.

	\begin{ex} \label{RandomEx}
If we pick $\vt_1 = \ldots = \vt_n = 5, \vt_{n+1} = \ldots = \vt_{2n} = 4, \vt_{2n+1} = 3$, for example, then the associated family of dilations defined by \eqref{GenDilDF} is given by
	\begin{equation}
	\left\{\begin{array}{rclrclrcl}
		D_r(Z) &=& r^{12} Z, &&&&&& \\
		D_r(Y_j) &=& r^{7} Y_j, &D_r(Y_{n + j}) &=& r^{8} Y_{n + j}, &D_r(X_{2n+1}) &=& r^{9} X_{2n+1}, \\
		D_r(X_j) &=& r^5 X_j, &D_r(X_{n + j}) &=& r^4 X_{n + j}, &D_r(Y_{2n+1}) &=& r^3 Y_{2n+1}.
	\end{array}\right. \nn 
	\end{equation}
for $j = 1, \ldots, n$. If we choose $\hdeg_0 = 13 \cdot 5 \cdot 8 \cdot 9 \cdot 7 = 32760$, which is the 13-fold least common multiple of $5, 7, 8, 9, 12$, then by \cite[Lem.~4.1.8]{FiRuMon}\footnote{or Example~\ref{GenericGradedEx}} the form
	\begin{align*}
		\RF = \sum_{j=1}^n e^\pi \hspace{1pt} X_j^{13104} + \pi^e \hspace{1pt} X_{n+j}^{16380} + \pi^\pi \hspace{1pt} Y_{2n+1}^{21840} + e^e \hspace{1pt} X_{2n+1}^{7280} + \frac{e}{\pi} \hspace{1pt} Y_j^{9360} - \frac{\pi}{e} \hspace{1pt} Y_{n+j}^{8190} + \frac{e^e}{\pi^\pi} \hspace{1pt} Z^{5460} \in \UEA{\HA{n}{2}}
	\end{align*}
is a positive homogeneous Rockland form of homogeneous degree $\hdeg =  65520$. The associated anharmonic oscillator $\HAO = d\pi_{\rp}(\RF)$ equals
	\begin{align*}
		\HAO^\rp = \sum_{j=1}^n &e^\pi \hspace{1pt} \bigl( \partial_{t_j} - \frac{1}{2} t_{n+j} \partial_{t_{2n+1}} \bigr)^{13104}
		+ \pi^e \hspace{1pt} \bigl( \partial_{t_{n+j}} + \frac{1}{2} t_j \partial_{t_{2n+1}} \bigr)^{16380}
		+ \pi^\pi \hspace{1pt} (2 \pi \Abs{\rp} \hspace{1pt} t_{2n+1})^{21840} \\
		+ &e^e \hspace{1pt} (2 \pi \Abs{\rp} \hspace{1pt} \partial_{t_{2n+1}}^{7280}
		+ \frac{e}{\pi} \hspace{1pt} (2 \pi \Abs{\rp} \hspace{1pt} t_j)^{9360}
		- \frac{\pi}{e} \hspace{1pt} (2 \pi \Abs{\rp} \hspace{1pt} t_{n+j})^{8190}
		+ \frac{e^e}{\pi^\pi} \hspace{1pt} (2 \pi \Abs{\rp})^{5460}.
	\end{align*}
We then have $N(\l, \pi_{\rp}, \RF) \asymp \Abs{\rp}^{-(2n+1)} \hspace{1pt} \l^{\frac{(2n+1)}{5460}}$.
	\end{ex}

	\begin{ex}
For $\vt_1 = \ldots = \vt_{2n} =1, \vt_{2n+1} = k \in \N$, the dilations defined by \eqref{GenDilDF} are given by
	\begin{equation}
	\left\{\begin{array}{rclrclrcl}
		D_r(Z) &=& r^{k + 2} Z, &&&&&& \\
		D_r(Y_j) &=& r^{k + 1} Y_j, &D_r(Y_{n + j}) &=& r^{k + 1} Y_{n + j}, &D_r(X_{2n+1}) &=& r^2 X_{2n+1}, \\
		D_r(X_j) &=& r X_j, &D_r(X_{n + j}) &=& r X_{n + j}, &D_r(Y_{2n+1}) &=& r^k Y_{2n+1}.
	\end{array}\right. \nn 
	\end{equation}
for $j = 1, \ldots, n$. By \cite[Lem.~4.1.8]{FiRuMon}, the form
	\begin{align*}
		\RF = \sum_{j=1}^n \left ( (-1)^k \hspace{1pt} X_j^{2k} + (-1)^k X_{n+j}^{2k} \right ) - \hspace{1pt} Y_{2n+1}^2 \in \UEA{\HA{n}{2}}
	\end{align*}
is a positive homogeneous Rockland form of homogeneous degree $\hdeg = 2k$. The associated anharmonic oscillator $\HAO = d\pi_{\rp}(\RF)$ equals
	\begin{align*}
		\HAO^\rp = \sum_{j=1}^n (-1)^k \left ( \bigl( \partial_{t_j} - \frac{1}{2} t_{n+j} \partial_{t_{2n+1}} \bigr)^{2k} + \bigl( \partial_{t_{n+j}} + \frac{1}{2} t_j \partial_{t_{2n+1}} \bigr)^{2k} \right ) + 4 \pi^2 \rp^2 \hspace{1pt} t_{2n+1}^2.
	\end{align*}
We then have $N(\l, \pi_{\rp}, \RF) \asymp \Abs{\rp}^{-(2n+1)} \hspace{1pt} \l^{\frac{(2n+1) (k+2)}{2k}}$.
	\end{ex}

	\begin{ex} \label{NoGoEx}
There are no $\vt_1, \ldots, \vt_{2n+1} \in \N$ so that for any integer $k > 1$ the form and associated operator
	\begin{align*}
		\RF &= - \sum_{j=1}^n \left ( X_j^2 + X_{n+j}^2 \right ) \pm \hspace{1pt} Y_{2n+1}^{2k}, \\
		\HAO^\rp &= - \SL_{\H} \pm (2 \pi \Abs{\rp})^{2k} \hspace{1pt} t_{2n+1}^{2k}
	\end{align*}
satisfy our criteria of Definition~\ref{DefHAO}. While it is straight-forward to check that for every integer $k > 1$, the operator
	\begin{align*}
		\RO_{+} := - \SL_{\H} + (2 \pi \Abs{\rp})^{2k} \hspace{1pt} t_{2n+1}^{2k}
	\end{align*}
is positive and satisfies the Rockland property \eqref{RockPr}, it is neither a homogeneous operator nor an operator whose leading order is a positive homogeneous Rockland operator, the leading order being $(2 \pi \Abs{\rp})^{2k} \hspace{1pt} t_{2n+1}^{2k}$, which itself does not satisfy \eqref{RockPr}.
The operator
	\begin{align*}
		\RO_{-} := - \SL_{\H} - (2 \pi \Abs{\rp})^{2k} \hspace{1pt} t_{2n+1}^{2k}
	\end{align*}
does not even necessarily satisfy the Rockland property.
	\end{ex}

\section{A Generalization to Graded $SI/Z$-Groups With $1$-Dimensional Center} \label{SecSpecEstSIZ}

In this section we show that the spectral estimates for the harmonic and anharmonic oscillators on $\R^n$ and the Heisenberg group $\H$, defined via the generic unitary irreducible representations of $\H$ and the Dynin-Folland group $\HG{n}{2}$, respectively, effortlessly extend to the case of operators $d\pi(\RF)$ which are defined via the generic representations $\pi \in \widehat{G}$ for a general graded group $G$, equipped with the canonical dilations, which satisfies
	\begin{itemize}
		\item[(SI)] $SI/Z(G) \neq \emptyset$, i.e., there exists a $\pi \in \widehat{G}$ which is square-integrable modulo the center $Z(G)$,
		\item[(Z)] $\dim(Z(G)) = 1$.
	\end{itemize}
We recall that by a celebrated theorem by Moore and Wolf~\cite{MoWo} the condition (SI) is satisfied if and only if almost all $\pi \in \widehat{G}$ are square-integrable modulo $Z(G)$, and if and only if the corresponding coadjoint orbits $\Orbit_\pi$ are affine subspaces of $\Lie{g}^*$ of codimension $\dim(Z(G))$. Such orbits are in particular flat. 
We will denote the set of such representations by $SI/Z(G)$. Thus, $SI/Z(G) \neq \emptyset$ implies $\widehat{G} \setminus SI/Z(G)$ has Plancherel measure zero.

The groups $\H$ and $\HG{n}{2}$ satisfy (SI) and (Z), and recently nilpotent groups with flat orbits, in particular graded ones, have been studied from the viewpoint of the Kohn-Nirenberg and Weyl-Pedersen quantizations in M\u{a}ntoiu and Ruzhansky~\cite{MaRu18}. Thus, (SI) and (Z) will be our standing assumptions throughout this section for a generic, but fixed graded group $G$ of dimension $\dimG$ and homogeneous dimension $\hdim$. The following additional working assumptions do not restrict the generality of our statements, yet we require them for the main statement of this section and its proof.

	\begin{ass} \label{AssGrGr}
We assume $G$ is gradable and satisfies (SI) and (Z). The group must therefore be of dimension $\dimG = 2n + 1$ for some $n \in \N$. Moreover, we fix a gradation of $\Lie{g}$ and denote it by $\Lie{g} = \oplus_{k = 1}^N \Lie{g}_k$ with $N \in \N$. It was shown in \cite[Lem.~4.16]{GrRo18} that the condition (Z) implies that $\Lie{g}_N = \Liez{g}$. Moreover, every basis for $\Lie{g}$ given as the union of bases for $\Lie{g}_k$ is a strong Malcev basis for $\Lie{g}$, which thus passes through $\Liez{g}$. We fix such a basis and denote it by $\{ Z, X_{2n}, \ldots, X_1 \}$. Since these basis vectors need not be eigenvectors of the dilation matrix $A$ in $D_r := \exp_{GL(n, \R)} \bigl( A \log(r) \bigr)$ for a given family of dilations, we assume that for a given gradation of $\Lie{g}$ the group $G$ is equipped with the canonical homogeneous dilations from Definition~\ref{CanDilGrGr}. This assumption may seem rather restrictive, but without any information about the eigenbasis of a given family of dilations it is not clear at all how the corresponding homogeneous balls look like, which are needed in order to apply ter Elst and Robinson's machinery.
	\end{ass}

With this at hand, the generalization of Proposition~\ref{SpecEstHAO} reads as follows.

	\begin{prop} \label{PropSpecEstSIZ}
Let $G$ be a graded group which satisfies $\mathrm{(SI)}$, $\mathrm{(Z)}$, and without loss of generality, the Assumptions~\ref{AssGrGr}. Let $\RF \in \UEA{\Lie{g}}$ be a positive Rockland form of homogeneous degree $\hdeg$ and let $\pi \in SI/Z(G)$. Denote by $\fd$ the formal dimension of $\pi$ and by $\hdim_{\Lie{z}}$ the homogeneous dimension of the center $\Liez{\Lie{g}}$. Then the number of eigenvalues of $d\pi(\RF)$, counted with multiplicities, which are less or equal $\l > 0$ is asymptotically given by
	\begin{align*}
		N(\l, \pi, \RF) \asymp \fd^{-1} \hspace{4pt} \l^{\frac{\hdim - \hdim_{\Lie{z}}}{\hdeg}}.
	\end{align*}
	\end{prop}

The proof is very similar to the proof of Proposition~\ref{SpecEstHHO}. We present it nevertheless.

	\begin{proof}
By the Kirillov theory, the generic unitary irreducible representations all have flat coadjoint orbits of co-dimension $1$, and the representations as well as the coadjoint orbits can be parameterized by representatives $l' \in \Liez{\Lie{g}}^* \setminus \{ 0 \} = \R Z^* \setminus \{ 0 \}$. Thus, given $\pi$, there exists an $l'_{\pi} \in \Liez{\Lie{g}}^* \setminus \{ 0 \}$ such that the uniquely determined coadjoint orbit corresponding to $\pi$ equals $\Orbit_{l'_{\pi}} = l'_{\pi} + \Rspan{X^*_{2n}, \ldots, X^*_1}$. 
The corresponding orbital measure is given by $$\om_{\pi} = \Abs{\Pf(l'_{\pi})}^{-1} \hspace{1pt} \Leb_{\R^{2n+1}} = \fd^{-1} \hspace{1pt} \Leb_{\R^{2n+1}},$$ where $\Leb_{\R^{2n+1}}$ denotes the $(2n+1)$-dimensional Lebesgue measure on the affine subspace $\Orbit_{l'_{\pi}} \subseteq \Lie{g}^*$. This follows directly from the direct integral representation $\Leb_{\Lie{g}^*} = \int^\oplus_{\Liez{\Lie{g}}^*} \om_{\pi_{l'}} d\Pla(\pi_{l'})$ of the Lebesgue measure on $\Lie{g}^*$ by the Plancherel theorem (cf.~\cite[Thm.~4.3.10]{CoGr}) and the fact that the Plancherel measure on $\Liez{\Lie{g}}^* \cong \Lie{g}^*/\coAd(G) \cong \widehat{G}$ is given by $d\Pla(\pi_{l'}) = \Abs{\Pf(l')} \hspace{1pt} dl'$.

Since $G$ is equipped with the canonical homogeneous dilations, we can use the quasi-norm $\qn{\Lie{g}^*}{\, . \,}^\infty$ defined by Proposition~\ref{QNGrGr} to compute
	\begin{align*}
		N_0(\l, \pi) &= \om_{\pi} \bigl( \bigl \{ l \in \Orbit_{l'_{\pi}} : \qn{\Lie{g}^*}{l}^\infty \leq \l \bigr \} \bigr) \\
		 &= \fd^{-1} \hspace{1pt} \Leb_{\R^{2n+1}} \bigl( \bigl \{ l \in \Orbit_{l'_{\pi}} : \bigl | l_{j_{\dimG_j}} \bigr |, \ldots, \bigl | l_{j_1} \bigr | \leq \l^j : j = 1, \ldots, N \bigr \} \bigr ) \\
		 &= \fd^{-1} \hspace{1pt} \hspace{1pt} \l^{\hdim - \hdim_{\Lie{z}}}.
	\end{align*}
Hence, by \cite[Thm.~4.1]{tERo} we have $N_0(\l^{1/\hdeg}, \pi) \asymp N(\l, \pi, \RF) \asymp \fd^{-1} \hspace{4pt} \l^{\frac{\hdim - \hdim_{\Lie{z}}}{\hdeg}}$.
	\end{proof}

	\begin{obs}
By \cite{GrRo18}, Propositions~4.17 and 4.18, $\Lie{g}$ possesses an ideal $\pid $ which is polarizing for all $l' \in \Liez{g}$. As shown in the proofs, there is no loss in generality, assuming that $\pid = \Rspan{Z, X_{2n}, \ldots, X_{n+1}}$. As a consequence, each $\pi_{l'} \in SI/Z(G)$ can be realized as acting on $\L{2}{\PID \rquo G}$, and one may assume that $\PID \rquo G$ carries the canonical quotient group structure and homogeneous structure. Hence, for each $\pi_{l'} \in SI/Z(G)$ the operator $d\pi_{l'}(\RF)$ from Proposition~\ref{PropSpecEstSIZ} acts as a differential operator of homogeneous degree $\hdeg$ on the homogeneous space $\PID \rquo G$.

For a concrete realization one can choose quite convenient coordinates, e.g., strong Malcev coordinates and exponential coordinates with respect to the basis $\{ Z, X_{2n}, \ldots, X_{n+1} \}$; the exponential coordinates work out in this case because $\PID$ is an ideal and the exponential coordinates quotient naturally under the canonical quotient map $\Lie{g} \cong \Rspan{Z, X_{2n}, \ldots, X_1} \to \pid \rquo \Lie{g} \cong \Rspan{X_n, \ldots, X_1}$.
	\end{obs}

\section{$L^\pt$-$L^\qt$-Multipliers and Sobolev Embeddings on Graded Groups} \label{LpLq}

In this section we present another instance of the usefulness of ter Elst and Robinson's techniques. 
We combine \cite[Thm.~4.1]{tERo} with the theory of spectral multipliers on locally compact groups developed in Akylzhanov and Ruzhansky~\cite{AkRu18} in order to prove the $L^\pt$-$L^\qt$-boundedness for spectral multipliers of positive Rockland operators on graded groups equipped with arbitrary homogeneous structures.
Specifically, we treat the heat semi-group $\bigl\{ e^{t \RO} \bigr \}_{t > 0}$ of a given Rockland operator $\RO$ and also obtain an alternative proof for the Sobolev embedding theorems first proved in Folland~\cite{Fo75} for stratified groups and, more generally, for graded groups in Fischer and Ruzhansky~\cite{FiRu17}.
Rockland operators which are homogeneous with respect to the homogeneous structures on the Dynin-Folland group characterized by Proposition~\ref{AnhDilDF} form special cases of Theorem~\ref{LpLqGrGr}.

Note that the theory developed in \cite{AkRu18} allows a much wider class of spectral multipliers on graded groups to be studied with the same methods we will use, but this would require a lengthier analysis and exceed the scope of this article.

We briefly set the notation in accordance with \cite{AkRu18}. 
For a more general introduction to the matter, we refer to von Neumann~\cite{vN49} and Dixmier~\cite{Dix81}.
Thus, we denote by $\vN_R(G)$ the right group von Neumann algebra of a graded group $G$ and by $\tr$ its canonical trace. We denote the spectral projections of the self-adjoint extension on $\L{2}{G}$ of a given positive Rockland operator $\RO = dR(\RF)$ by $\{ E_{\l}(\RO) \}_{\l \in \Sp(\RO)}$, and by $E_{(0, s)}(\RO)$ the spectral projection corresponding to the interval $(0, s)$. We say that a closed linear operator is affiliated with $\vN_R(G)$ if it commutes with the commutant ${\vN_R(G)}^! = \vN_L(G)$. This includes the case of the Dynin-Folland sub-Laplacian $\RO = -\SL_{\HG{n}{2}}$, which gave rise to the harmonic oscillator on $\H$ in Section~\ref{SectionHHO}.


To start with, we recall the definition of (inhomogeneous) Sobolev spaces on graded groups (cf.~\cite[Def.~4.5]{FiRu17}).

	\begin{dfn} \label{DefSobSp}
Let $G$ be a graded group equipped with an arbitrary but fixed homogeneous structure. Let $P \in \UEA{\Lie{g}}$ be a positive Rockland form of homogeneous degree $\hdeg$, let $\pt \in [1, \infty)$, and let $\st \in \R$. Then the (inhomogeneous) Sobolev space $\LW{\pt}{G}{\st}$ is defined as the subspace of $\SDG{G}$ obtained by the completion of $\SFG{G}$ with respect to the Sobolev norm
	\begin{align*}
		\Norm{\LW{\pt}{G}{\st}}{f} := \Norm{\L{\pt}{G}}{\bigl( I + \RO \bigr)^{\frac{\st}{\hdeg}}}, \hspace{10pt} f \in \SFG{G}.
	\end{align*}
	\end{dfn}

Note that by \cite[Thm.~4.4.20]{FiRuMon}, the Sobolev space $\LW{\pt}{G}{\st}$ does not depend on the specific Rockland operator $\RO$.

The main result of this section is the following very general theorem, for which we can give a surprisingly short proof by combining the deep results in \cite{tERo} and \cite{AkRu18}.

	\begin{thm} \label{LpLqGrGr}
Let $G$ be a graded group of topological dimension $\dimG$ equipped with a homogeneous structure. Let $\vt_1, \ldots, \vt_\dimG$ denote the associated weights and let $\hdim = \vt_1 + \ldots + \vt_\dimG$ denote the homogeneous dimension of $G$. Let $P \in \UEA{\Lie{g}}$ be a positive Rockland form of homogeneous degree $\hdeg$
and let $\alpha := \frac{\hdim}{\hdeg}$. Then the Rockland operator $\RO := dR(\RF)$ satisfies
	\begin{align*}
		\tr \bigl( E_{(0, \l)}(\RO) \bigr) \asymp \l^\alpha,
	\end{align*}
and for all $1 < \pt \leq 2 \leq \qt < \infty$ there exists a constant $C = C(\alpha, \pt, \qt) > 0$ such that
	\begin{align}
		\Norm{\L{\pt}{G} \to \L{\qt}{G}}{e^{-t \RO}} \leq C \hspace{1pt} t^{-\alpha \bigl( \frac{1}{\pt} - \frac{1}{\qt} \bigr)}, \hspace{10pt} t \to \infty, \nn 
	\end{align}
and such that
	\begin{align*}
		\Norm{\L{\qt}{G}}{f} \leq C \hspace{1pt} \Norm{\L{\pt}{G}}{(1 + \RO)^\gamma f} \hspace{10pt} \mbox{ if } \hspace{10pt} \gamma \geq \alpha \Bigl( \frac{1}{\pt} - \frac{1}{\qt} \Bigr).
	\end{align*}
In particular, this yields the continuous embeddings
	\begin{align*}
		\LW{\pt}{G}{\st_1} \subseteq \LW{\qt}{G}{\st_2}.
	\end{align*}
for $\st_1, \st_2 \in \R$ with $\st_1 - \st_2 = \frac{\hdim}{\hdeg} \Bigl( \frac{1}{\pt} - \frac{1}{\qt} \Bigr)$.
	\end{thm}

We prove the above estimates by making use of \cite{AkRu18}, Corollaries~8.1 and 8.2, respectively. Note that the statements are akin to \cite[Thm.~6.1]{AkRu18}, a generalized version of H\"{o}rmander's multiplier theorem for the setting of separable unimodular locally compact groups, which provides an asymptotic bound from above for the $\BL(\L{\pt}{G},\L{\qt}{G})$-operator norm of Fourier multipliers. 
The Sobolev embedding in the above estimates give a different justification for those established in \cite{FiRu17}, for the corresponding ranges of indices: but see \cite{FiRuMon,FiRu17} for a more comprehensive treatment of Sobolev spaces on graded groups and their embeddings.


	\begin{proof} [Proof of Theorem~\ref{LpLqGrGr}]
Let $\varphi_1, \varphi_2: \R^+ \to \R$ be defined by $\varphi_1(u) := e^{-tu}$ and $\varphi_2(u) := (1+u)^{- \gamma}$. Since ${\vN_R(G)}^! = \vN_L(G)$ and the left-invariant operator $\RO = dR(\RF)$ clearly commutes with the left regular representation, $\RO$ and its spectral multipliers $\varphi_1(\RO) = e^{-t \RO}$ and $\varphi_2(\RO) = (1 - \RO)^\gamma$ are affiliated with $\vN_R(G)$.
Now, in order to apply \cite[Cor.~8.1]{AkRu18} and \cite[Cor.~8.2]{AkRu18}, we need to provide an estimate of the form
	\begin{align*}
		\tr \bigl( E_{(0, \l)}(\RO) \bigr) \lesssim \l^\alpha, \hspace{10pt} \l \to \infty,
	\end{align*}
for $\alpha = \frac{\hdim}{\hdeg}$. The estimate is due to the following argument: On the one hand, by Dixmier~\cite[p.~225, Thm.~1]{Dix81}, we can decompose the spectral projection of $\RO$ as 
	\begin{align}
		\tr \bigl( E_{(0, \l)}(\RO) \bigr) &= \int^\oplus_{[\pi] \in \widehat{G}} \tr \Bigl( E_{(0, \l)} \bigl( \pi(\RO) \bigr) \Bigr) \, d\Pla(\pi). \label{SpecInt}
	\end{align}
On the other hand, we employ the decomposition
	\begin{align}
		\Leb_{\Lie{g}^*} = \int^\oplus_{[\pi] \in \widehat{G}} \hspace{4pt} \om_{\pi} \, d\Pla(\pi) \label{DecLeb}
	\end{align}
of the Lebesgue measure on $\Lie{g}^*$ into a direct integral of the orbital measures $\om_{\pi}$. The decomposition follows from the Plancherel theorem for nilpotent groups, and the Fubini-Tonelli-like formula \eqref{DecLeb} is a consequence of the actual proof of the Plancherel theorem (see, e.g., \cite[Thm.~4.3.10]{CoGr} and the subsequent formula and explanation on page~154). As the proof shows, the Lebesgue measure splits up as the direct sum
	\begin{align}
		\Leb_{\Lie{g}^*} = \int^\oplus_{[\pi] \in \widehat{G}_{gen}} \hspace{4pt} \om_{\pi} \, d\Pla(\pi) \label{DecLebGO}
	\end{align}
of the orbital measures corresponding to the generic representations. Since the subset of generic representations $\widehat{G}$ has full Plancherel measure in $\widehat{G}$ (cf.~\cite[Thm.~4.3.16]{CoGr}), \eqref{DecLebGO} already implies \eqref{DecLeb}. By the proof of the Plancherel theorem, the Lebesgue measure on $\Lie{g}^*$ in \eqref{DecLeb} is normalized with respect to a strong Malcev basis which allows the parameterization of the generic coadjoint orbits. Let us denote by $\Phi$ the linear change of variables which maps this basis onto the basis which diagonalizes the matrix $A$ of the dilations $D_r := \exp_{GL(n, \R)} \bigl( A \log(r) \bigr)$. 
Now, if $B_\l(0)$ denotes the ball of radius $\l$ around $0 \in \Lie{g}^*$ defined by the homogeneous quasi-norm $\qn{\Lie{g}^*, A}{\, . \,}^{\infty}$ from Proposition~\ref{QNGrGr}, then by the proposition and by \cite[Thm.~4.1]{tERo} we have
	\begin{align}
		\tr \Bigl ( E_{(0, \l)} \bigl( \pi(\RF) \bigr) \Bigr ) \asymp N(\l, \pi, \RF) \asymp \om_{\pi}\bigl( B_{\l^\frac{1}{\hdeg}}(0) \cap \Orbit_\pi \bigr), \label{AbstractSpecEst}
	\end{align}
uniformly for all $\l > 0$ and all representations $\pi \in \widehat{G}$. Combing \eqref{DecLeb} and \eqref{AbstractSpecEst}, we obtain
	\begin{align}
		\tr \bigl( E_{(0, \l)}(\RO) \bigr) &\asymp \int^\oplus_{[\pi] \in \widehat{G}} \om_{\pi} \bigl( B_{\l^\frac{1}{\hdeg}}(0) \cap \Orbit_\pi \bigr) \, d\Pla(\pi)
		= \Leb_{\Lie{g}^*} \bigl( B_{\l^\frac{1}{\hdeg}}(0) \bigr) \nn \\
		&= \Abs{\det(\Phi)}^{-1} \hspace{2pt} 2^\dimG \hspace{2pt} \l^{\vt_1/\hdeg} \cdots \l^{\vt_\dimG/\hdeg} \asymp \l^{\frac{\hdim}{\hdeg}}. \label{SpecBdRO}
	\end{align}
An application of \cite[Cor.~8.1]{AkRu18} and \cite[Cor.~8.2]{AkRu18}, respectively, finally yields the announced estimates.
	\end{proof}

	\begin{exs}
Theorem~\ref{LpLqGrGr} holds for the negative sub-Laplacian $\RO := -\SL_G$ of every stratified group $G$, e.g., the Heisenberg group $\H$ with $\hdim = 4n+2$, the Dynin-Folland group $\HG{n}{2}$ with $\hdim = 6n+6$, the Engel group (cf.~\cite[Ex.~1.3.10]{CoGr}) with $\hdim = 7$, etc.
	\end{exs}


As a corollary we get the following specific statement for the homogeneous structures on the Dynin-Folland group $\HG{n}{2}$ characterized by Proposition~\ref{AnhDilDF}.

	\begin{prop} \label{GenLpLqDF}
Let the Dynin-Folland group $\HG{n}{2}$ be equipped with the dilations~\eqref{GenDilDF} determined by the weights $\vt_1, \ldots, \vt_{2n+1} \in \N$, and denote by $\hdim$ the associated homogeneous dimension of $\HG{n}{2}$ equal to $(2n+2) \hspace{1pt} (\vt_1 + \vt_{n+1} + \vt_{2n+1})$. Let $P \in \UEA{\HA{n}{2}}$ be a positive Rockland form which is homogeneous of degree $\hdeg$ with respect to the homogeneous structure under consideration. Moreover, let
	\begin{align*}
		\alpha := \frac{\hdim}{\hdeg} = \frac{(2n+2)(\vt_1 + \vt_2 + \vt_3)}{\hdeg}.
	\end{align*}
Then the Rockland operator $\RO := dR(\RF)$ satisfies
	\begin{align*}
		\tr \bigl( E_{(0, \l)}(\RO) \bigr) \asymp \l^\alpha
	\end{align*}
and for all $1 < \pt \leq 2 \leq \qt < \infty$ there exists a constant $C = C(\alpha, \pt, \qt) > 0$ such that
	\begin{align*}
		\Norm{\L{\pt}{\HG{n}{2}} \to \L{\qt}{\HG{n}{2}}}{e^{-t \RO}} \leq C \hspace{1pt} t^{-\alpha \bigl( \frac{1}{\pt} - \frac{1}{\qt} \bigr)}, \hspace{10pt} t \to \infty,
	\end{align*}
and such that
	\begin{align*}
		\Norm{\L{\qt}{\HG{n}{2}}}{f} \leq C \hspace{1pt} \Norm{\L{\pt}{\HG{n}{2}}}{(1 + \RO)^\gamma f} \hspace{10pt} \mbox{ if } \hspace{10pt} \gamma \geq \alpha \Bigl( \frac{1}{\pt} - \frac{1}{\qt} \Bigr).
	\end{align*}
In particular, this yields the continuous embeddings
	\begin{align*}
		\LW{\pt}{\HG{n}{2}}{\st_1} \subseteq \LW{\qt}{\HG{n}{2}}{\st_2}.
	\end{align*}
for $\st_1, \st_2 \in \R$ with $\st_1 - \st_2 = \frac{\hdim}{\hdeg} \Bigl( \frac{1}{\pt} - \frac{1}{\qt} \Bigr)$. 
	\end{prop}


We conclude this section with an alternative proof of Proposition~\ref{GenLpLqDF} for the case $G = \HG{n}{2}$, which is exemplary for the class of graded $SI/Z$-groups with $1$-dimensional center. The proof takes an intermediate step by estimating directly the integral in \eqref{SpecBdRO} for $G = \HG{n}{2}$ by using the estimates derived in Subection~\ref{Subsection_HAO}.

	\begin{proof} [Alternative Proof of Proposition~\ref{GenLpLqDF}]
We give an alternative proof only for the bound \eqref{SpecBdRO} since the rest follows by applying \cite[Cor.~8.1]{AkRu18} and \cite[Cor.~8.2]{AkRu18}, respectively. Our argument here goes as follows: Instead of using the shortcut \eqref{SpecBdRO}, we compute the integral \eqref{SpecInt} directly via the spectral estimates from Proposition~\ref{SpecEstHAO}. The unitary dual $\widehat{\mathbf{H}}_{n, 2}$ can be identified with $\Liez{\HA{n}{a}}^* = \R Z^*$; this is justified by the fact that $\widehat{\mathbf{H}}_{n, 2}$ equals $SI/Z(\HG{n}{2})$ up to a set of vanishing Plancherel measure.\footnote{If this argument were followed for non-$SI/Z$-groups $G$, then the spectral estimates for all $\pi \in \widehat{G}$ would have to be employed.}

Thus, given $\rp \in \R \setminus \{ 0 \}$, we have
	\begin{equation} \label{}
	\begin{array}{rclcl}
		N(\l, \pi_{\rp}, \RF) & \asymp & \om_{\pi_{\rp}} \bigl (B_{\l^{1/2}}(0) \bigr ) &=& \Abs{\Pf(\rp Z^*)} \hspace{3pt} \Leb_{\R^{4n+2}} (B_{\l^{1/2}}(0) \bigr ) \\
		&=& \Abs{\rp}^{-(2n+1)} \hspace{1pt} \l^{\frac{\hdim - \hdim_{\Lie{z}}}{\hdeg}} &=& \Abs{\rp}^{-(2n+1)} \hspace{1pt} \l^{\frac{(2n + 1) (\vt_1 + \vt_{n+1} + \vt_{2n+1})}{\hdeg}}.
	\end{array}
	\end{equation}
Now, for $l = l_z Z^* + l_{y_1} Y^*_1 + \ldots + l_{x_1} X^*_1 \in B_{\l^{1/\hdeg}}(0)$ we have
	\begin{align*}
		B_{\l^{1/\hdeg}}(0) \cap \Orbit_{\rp Z^*} = \emptyset &\Leftrightarrow \qn{\HA{n}{2}^*, A}{l_z Z^*}^{\infty} \leq \l^{\frac{1}{\hdeg}} < \qn{\HA{n}{2}^*, A}{\rp Z^*} \\
		&\Leftrightarrow \Abs{l_z}^{\frac{1}{\hdim_{\Lie{z}}}} \leq \l^{\frac{1}{\hdeg}} < \Abs{\rp}^{\frac{1}{\hdim_{\Lie{z}}}} \\
		&\Leftrightarrow \Abs{l_z} \leq \l^{\frac{\hdim_{\Lie{z}}}{\hdeg}} < \Abs{\rp},
	\end{align*}
where $\hdim_{\Lie{z}} = \vt_1 + \vt_{n+1} + \vt_{2n+1}$.
So,
	\begin{align*}
		\tr \bigl( E_{(0, \l)}(\RO) \bigr) &\asymp \int_{\R \setminus \{ 0 \}} N(\l, \pi_{\rp}, \RF) \Abs{\rp}^{-(2n+1)} \,d\rp \\
		&\asymp \int^{\l^{\frac{\hdim_{\Lie{z}}}{\hdeg}}}_{-\l^\frac{\hdim_{\Lie{z}}}{\hdeg}} \Abs{\rp}^{-(2n+1)} \hspace{1pt} \l^{\frac{\hdim - \hdim_{\Lie{z}}}{\hdeg}} \Abs{\rp}^{-(2n+1)} \,d\rp \\
		&= 2 \hspace{1pt} \l^{\frac{\hdim_{\Lie{z}}}{\hdeg}} \hspace{1pt} \l^{\frac{\hdim - \hdim_{\Lie{z}}}{\hdeg}} 
		\asymp \l^{\frac{\hdim}{\hdeg}}.
	\end{align*}
This completes our proof.
	\end{proof}

	\begin{rem} 
This alternative proof of Theorem~\ref{LpLqGrGr} for $G = \HG{n}{2}$ does in general not carry over easily to all graded groups, in particular not to non-$SI/Z$-groups. In the above proof we have made use of the fact that the homogeneous volume of $B_{\l^{1/\hdeg}}(0) \cap \Orbit_{\rp Z^*}$ could be computed easily. Our computation is facilitated by the fact that the central basis vector $Z$ is an eigenvector of the dilation matrix $A$.\footnote{In general, one has no information about how the elements in $Z(G)$ scale and different vectors in $Z(G)$ will in general scale differently.} This need not be the case, but at least the orbital measure is very simple for $SI/Z$-groups, namely a scaled version of the Lebesgue measure.

Let us highlight the additional difficulty of determining the orbital measures for a concrete example of a non-$SI/Z$-group, the $3$-step nilpotent Engel group. The Engel Lie algebra $\Lie{g}$ is defined as $\R^4$ equipped with the Lie bracket which for the standard basis $\{ X_1, \ldots, X_4 \}$ is defined by
	\begin{align*}
		[X_1, X_2] = X_3, \hspace{5pt} [X_1, X_3] = X_4.
	\end{align*}
The Engel Lie algebra is $3$-step nilpotent and possesses a natural stratification $\Lie{g} = \oplus_{k = 1}^3 \Lie{g}_k$ with
	\begin{align*}
		\Lie{g}_3 = \R X_4, \hspace{5pt} \Lie{g}_2 = \R X_3, \hspace{5pt} \Lie{g}_1 = \Rspan{X_1, X_2}.
	\end{align*}
The coadjoint orbits of the corresponding connected, simply connected nilpotent group $G$, the Engel group, are classified as follows, see~\cite[Ex.~2.2.7]{CoGr} or \cite[Ch.3~\S~3]{Ki04}:
	\begin{itemize}
		\item[(i)] $\Orbit_{\delta X^*_4 + \beta X^*_2} = \bigl \{ \delta X^*_4 + t X^*_3 + (\beta + \frac{t^2}{2 \delta}) X^*_2 + s X^*_1 : s,t \in \R \bigr \}$ if $\delta \neq 0$,
		\item[(ii)] $\Orbit_{\gamma X^*_3} = \gamma X^*_3 + \Rspan{X^*_1, X^*_2}$ if $\gamma \neq 0$,
		\item[(iii)] $\Orbit_{\alpha X^*_1 + \beta X^*_2} = \{ \alpha X^*_1 + \beta X^*_2 \}$ for $\alpha, \beta \in \R$.
	\end{itemize}

Thus, $G$ has no square-integrable representations modulo $Z(G) = \exp_G(\R X_4)$, only modulo the projective kernel $\exp_G(\R X_3 \oplus \R X_4)$ and the projective representations corresponding to the orbits $\Orbit_{\gamma X^*_3}$ are projective Schr\"{o}dinger representations $\rho_{\gamma}$. Nevertheless, the Plancherel measure $\mu$ is concentrated on $\{ \pi_{\delta, \beta} : \delta \neq 0, \beta \in \R \}$ and $\mu$ can be identified with $\Abs{\delta} d\delta \, d\beta$. The orbits $\Orbit_{\delta X^*_4 + \beta X^*_2}$ are $3$-dimensional parabolic cylinders which fill up the $3$-dimensional affine hyperplane $\delta X^*_4 + \Rspan{X^*_1, X^*_2, X^*_3}$ whenever $\delta \neq 0$.

Hence, for the negative sub-Laplacian $-\SL_G$ related to the canonical homogeneous structure from Definition~\ref{CanDilGrGr} and the canonical quasi-norm $\qn{\Lie{g}^*}{\, . \,}^\infty$ from Proposition~\ref{QNGrGr} the estimate
	\begin{align}
		\tr \bigl( E_{(0, \l)}(-\SL_G) \bigr) \asymp \int^\oplus_{[\pi] \in \widehat{G}} \om_{\pi} \bigl( B_{\l^\frac{1}{\hdeg}}(0) \cap \Orbit_\pi \bigr) \, d\Pla(\pi), \label{SpecBdEngel}
	\end{align}
which is essentially due to \cite[Thm.~4.1]{tERo}, implies that, regardless of the concrete orbital measure $\om_{\pi}$, a representation $\pi \in \widehat{G}$ does not contribute to \eqref{SpecBdEngel} unless $[\pi] = [\pi_{\alpha, \beta}]$ and $B_{\l^{1/2}}(0) \cap \Orbit_{\delta X^*_4 + \beta X^*_2} = \emptyset$;
as in the alternative proof of Theorem~\ref{LpLqGrGr} for $G = \HG{n}{2}$, one need actually not know the values of $\om_{\pi_{\delta, \beta}} \bigl( B_{\l^\frac{1}{\hdeg}}(0) \cap \Orbit_{\delta X^*_4 + \beta X^*_2} \bigr)$ for $\delta \in \R \setminus \{ 0 \}, \beta \in \R$, in order to know that \eqref{SpecBdEngel} sums up to $\Leb_{\Lie{g}^*} \bigl( B_{\l^\frac{1}{\hdeg}}(0) \bigr) = \l^{\frac{\hdim}{\hdeg}}$. For general $G$ and $\pi \in \widehat{G}$ it may already be very difficult to determine the arclength $\om_{\pi} \bigl( B_{\l^\frac{1}{\hdeg}}(0) \cap \Orbit_{\pi} \bigr)$, but even in our case determining the limits of the integral is not trivial and requires case distinctions depending on $\beta$ and $\lambda$; integrating the arclengths as functions of $\beta$ in order to determine the numerical value of \eqref{SpecBdEngel} proves to be another technical challenge, very much in contrast to the simple case of flat orbits above.
	\end{rem}

\section*{Acknowledgments}

Michael Ruzhansky was supported in parts by the FWO Odysseus Project, by the EPSRC grant EP/R003025/1 and by the Leverhulme
Grant RPG-2017-151.

David Rottensteiner was supported by the Austrian Science Fund (FWF) project [I~3403], awarded to Stephan Dahlke, Hans Feichtinger and Philipp Grohs.

The authors would like to thank V\'eronique Fischer and Julio Delgado for discussions, which have led also to the recent papers \cite{FiRoRu} and \cite{ChDeRu18}, dealing also with the Dynin-Folland group and the Heisenberg-modulation spaces, and with the Euclidean anharmonic oscillators via the Weyl-H\"ormander theory, respectively.

\bibliographystyle{alphaabbr}
\bibliography{Bib_HAO}

\end{document}